\documentclass{amsart}

\RequirePackage{color}

\def\smallskip{\vskip\smallskipamount}
\def\medskip{\vskip\medskipamount}
\def\bigskip{\vskip\bigskipamount}

\usepackage{graphicx}
\usepackage{amsmath,amsthm,bm,mathrsfs,amscd,tikz}
\usepackage{ytableau}
\usepackage{comment}
\usepackage{mathdots}
\usepackage[all]{xy}
\usepackage{graphicx}
\usepackage[colorinlistoftodos]{todonotes}
\usepackage{enumerate}
\usepackage{feynmp-auto}
\usepackage[english]{babel}
\usepackage[utf8]{inputenc}
\usepackage{float}
\usepackage{tikz}
\usepackage{tikz-cd}
\usepackage{amssymb}
\usepackage{graphicx}
\usepackage{mathtools}
\usetikzlibrary{decorations.pathmorphing}

\newtheoremstyle{thmstyle}{}{}{\itshape}{}{\bfseries}{ }{5pt}{}
\newtheoremstyle{exstyle}{}{}{}{}{\bfseries}{ }{5pt}{}
\newtheoremstyle{defstyle}{}{}{}{}{\bfseries}{ }{5pt}{}
\newtheoremstyle{remstyle}{}{}{}{}{\bfseries}{ }{5pt}{}

\theoremstyle{thmstyle}
\newtheorem{thm}{Theorem}[section]
\newtheorem{theorem}[thm]{Theorem}

\newtheorem{lemma}[thm]{Lemma}

\newtheorem{proposition}[thm]{Proposition}

\newtheorem{corollary}[thm]{Corollary}

\theoremstyle{exstyle}

\newtheorem{example}[thm]{Example}

\theoremstyle{defstyle}

\newtheorem{definition}[thm]{Definition}
\newtheorem{def-prop}[thm]{Definition-Proposition}

\theoremstyle{remstyle}

\newtheorem{remark}[thm]{Remark}

\theoremstyle{remstyle}

\renewcommand{\atop}[2]{\genfrac{}{}{0pt}{}{#1}{#2}}

\newcommand{\str}{\operatorname{str}}

\newcommand{\cohook}{\operatorname{cohook}}
\newcommand{\kiss}{\operatorname{kiss}}
\newcommand{\Kiss}{\operatorname{Kiss}}
\newcommand{\Hom}{\operatorname{Hom}}
\newcommand{\Fac}{\operatorname{Fac}}
\newcommand{\Ext}{\operatorname{Ext}}
\newcommand{\cWc}{{{}_cW_c}}
\newcommand{\cW}{{{}_cW}}
\newcommand{\Wc}{W_c}
\newcommand{\Wh}{W_{/h}}
\newcommand{\hW}{{}_{h\backslash}W}

\DeclareMathOperator*{\Coker}{Coker}
\DeclareMathOperator*{\Image}{Im}		
\DeclareMathOperator*{\modu}{mod}
\DeclareMathOperator*{\proj}{proj}
\DeclareMathOperator*{\add}{add}
\DeclareMathOperator*{\Du}{D}
\newcommand{\keg}{\left[{\atop{Y''}{ X''}} {\scriptstyle Z}{\atop{Y'}  {X'}}\right]}
\newcommand{\kbp}{K^b(\proj A)}
\renewcommand{\Mc}{\operatorname{Mc}}
\newcommand{\jm}{j_{\rm max}}
\newcommand{\im}{i_{\rm max}}

\title{On the Combinatorics of Gentle Algebras}
\author[Brüstle, Douville, Mousavand, Thomas, Y\i ld\i r\i m]{Thomas Brüstle, Guillaume Douville, Kaveh Mousavand, Hugh Thomas, Emine Y\i ld\i r\i m}

\begin{document}

\begin{abstract}{For $A$ a gentle algebra, and $X$ and $Y$ string modules, we construct a combinatorial basis for $\Hom(X,\tau Y)$.  We use this to describe support $\tau$-tilting modules for $A$.  We give a combinatorial realization of maps in both directions realizing the bijection between support $\tau$-tilting modules and functorially finite torsion classes.  We give an explicit basis of $\Ext^1(Y,X)$ as short exact sequences.  We analyze several constructions given in a more restricted, combinatorial setting by McConville \cite{McC}, showing that many but not all of them can be extended to general gentle algebras.}
\end{abstract} 
\maketitle

\section{Introduction}
In this paper, we study the combinatorics of gentle algebras.
Suppose that $X$ and $Y$ are string modules for a gentle
algebra $A=kQ/I$.  (Terminology not explained in the introduction will be defined in the next section.)  An explicit basis of $\Hom(X,Y)$ has been known for a long time \cite{CB}.  We give an analogous construction for $\Hom(X, \tau_A Y)$.  

Our construction proceeds by embedding $A=kQ/I$ into a larger gentle algebra $\hat A=k\hat Q/\hat I$, which we call the \emph{fringed algebra} of $A$.  The reason for embedding $A$ in $\hat A$ is that, although $\tau_A$ is well-understood thanks to work of Butler and Ringel \cite{BR}, its behaviour is somewhat complicated.  It turns out that the behaviour of $\tau_{\hat A}$ on $\modu A$ is more uniform and thus easier to analyze than that of $\tau_A$.  

Using our construction, in Section \ref{maximal-non-kissing}, we give an explicit description in terms of combinatorics of strings of the support $\tau$-tilting modules for $A$, as certain \emph{maximal non-kissing collections} of strings.  Support $\tau$-tilting modules are in bijection with functorially finite torsion classes; in Section \ref{comb}, we show explicitly in terms of the combinatorics of strings how to pass from a maximal non-kissing collection to its associated functorially finite torsion class and how to return from a functorially finite torsion class to the maximal non-kissing collection.  
If there are only finitely many functorially finite torsion classes, we give, in Section \ref{torposet}, a combinatorial description of the poset of functorially finite torsion classes as a combinatorially-defined poset on the maximal non-kissing collections.  

By Auslander-Reiten duality, $\Ext^1(Y,X)$ is dual to a quotient of $\Hom(X, \tau Y)$.  In Section \ref{Ext}, we use our construction of $\Hom(X,\tau Y)$ to find a basis for $\Ext^1(Y,X)$, and we realize that basis as a collection of extensions of $Y$ by $X$.  

Our paper was inspired throughout by work of McConville \cite{McC}.  He was working in a combinatorial context, studying certain generalizations of the Tamari lattice, but he uncovered many phenomena which extend practically verbatim to arbitrary gentle algebras.  

One phenomenon which does extend beyond McConville's setting, but not to all
gentle algebras, is his observation that when two maximal non-kissing collections are related by a single mutation, the exchanged strands kiss exactly once.  In Section \ref{uniqueness}, we give an example (originally due to
\cite{GLS}) showing that this does not hold for all gentle algebras.  We then show that it holds for all gentle algebras such that all $\tau$-rigid modules are
bricks.

While working on this project, we were informed of two other papers which overlap with our results.  \cite{PPP} also 
constructs the fringed algebra (which they call the ``blossoming algebra").  They apply this to realize and study the support $\tau$-tilting fan for gentle algebras.  \cite{CPS} studies $\Hom$ groups in the derived category of representations of a gentle algebra.  From this, they also deduce an explicit description of the extensions between string modules.  Another recent paper which is relevant is \cite{EJR}: they also study $\Hom(X,\tau Y)$, but because they work in the more general setting of string algebras, the combinatorics which they analyze is more complicated.  

\subsection*{Acknowledgements} T.B. was partially supported by an NSERC Discovery Grant.  G.D. was partially supported by an NSERC Alexander Graham Bell scholarship.  K.M. and E.Y. were partially supported by ISM scholarships.  H.T. was partially supported by NSERC and the Canada Research Chairs program.  This paper was developed in the context of the LaCIM Representation Theory Working Group, and benefited from discussions with its other participants, including in particular Mathieu Guay-Paquet and Amy Pang.



\section{Preliminaries and Background}

\subsection*{Notations and Conventions}

Throughout this paper $k$ denotes an arbitrary field. 
A quiver $Q=(Q_0,Q_1,s,e)$ is a directed graph, which we always assume to be finite and connected, with $Q_0$ the vertex set, $Q_1$ the set of arrows, and $s, e:Q_1\to Q_0$ two functions which respectively send each arrow $\gamma \in Q_1$ to its source $s(\gamma)$ and its target $e(\gamma)$. We use lower case Greek letters $\alpha$, $\beta$, $\gamma$, $\ldots$ to denote arrows of $Q$. 

A \emph{path of length} $m\geq 1$ in $Q$ is a finite sequence of arrows $\gamma_{m}\cdots\gamma_{2}\gamma_{1}$ where $s(\gamma_{j+1})=e(\gamma_{j})$, for every $1 \leq j \leq m-1$. By $kQ$ we denote the \emph{path algebra} of $Q$.

A \emph{zero (monomial) relation} in $Q$ is given by a path of the form $\gamma_{m}\cdots\gamma_{2}\gamma_{1}$, where  $m \geq 2$.
A two sided ideal $I$ is called \emph{admissible} if $R_Q ^m \subset I \subseteq R_Q^2$ for some $m\geq 2$, where $R_Q$ denotes the arrow ideal in $kQ$.  
In what follows, $I$ always denotes an admissible ideal generated by a set of zero relations in $Q$.  We use capital letters $A$ and $B$ to denote quotients of $kQ$ by such ideals.

\subsection*{Gentle algebras}\label{String Algebra}

$B = kQ/I$ is called a \emph{string algebra} if the following conditions hold:

\begin{enumerate}[(S1)]
\item There are at most two incoming and two outgoing arrows at every vertex of $Q$.
\item For every arrow $\alpha$, there is at most one $\beta$ and one $\gamma$ such that $\alpha\beta \notin I$ and $\gamma\alpha \notin I$.
\end{enumerate}

A string algebra is called \emph{gentle} if the following additional conditions are satisfied:

\begin{enumerate}[(G1)]
\item The ideal $I$ is quadratic, i.e. there is a set of monomials of length two that generate $I$.
\item For every arrow $\alpha$, there is at most one $\beta$ and one $\gamma$ such that $0\ne\alpha\beta \in I$ and $0\ne\gamma\alpha \in I$.
\end{enumerate}

\subsection*{Strings and Bands} To introduce the notion of string, we need the following definitions:

For a given quiver $Q$, let $Q^{op}$ be the \emph{opposite} quiver, obtained from $Q$ by reversing its arrows. We consider the arrows in $Q^{op}$ as the formal inverses of arrows in $Q$ and denote them by $\gamma^{-1}$, for every $\gamma \in Q$.  

Let $Q^s$ be the \emph{double} quiver of $Q$, which has $Q_0$ as the vertex set and $Q_1 \cup Q^{op}_1$ as the arrow set.  A \emph{string in $Q$} is a path $C = \gamma_n\cdots\gamma_1$  in $Q^s$ with the following properties:

\begin{enumerate}
\item[(A1)] No $\gamma_i$ is followed by its inverse.
\item[(A2)] Neither $C$ nor $C^{-1}$ contains a subpath in $I$.
\end{enumerate}

In such a case, $C = \gamma_n\cdots\gamma_1$ is called a string of length $n$ which starts at $s(C)=s(\gamma_1)$ and ends at $e(C)=e(\gamma_n)$. In addition, to every vertex $x\in Q_0$ we associate a zero-length string $e_{x}$ starting and
ending at $x$ which is its own formal inverse.

The set of all strings in $Q$ is denoted by $\str(Q)$. It is clear that $\str(Q)$ really depends on $Q$ and $I$, but we suppress $I$ from the notation, because it will be clear from context.

We call $C = \gamma_n\cdots\gamma_1$ a \emph{direct} string if $\gamma_i \in Q_1$ for every $1 \leq i \leq n$, and dually $C$ is called \emph{inverse} if $C^{-1}$ is direct.
A \emph{band} in $Q$ is a string $C = \gamma_n\cdots\gamma_1$ such that $C^n$ is well-defined for each $n \in \mathbb{N}$ and, furthermore, $C$ itself is not a strict power of another string of a smaller length.

\begin{definition} \label{stringdiagram} For every string $C = \gamma_n\cdots\gamma_1$, the associated \emph{diagram} of $C$ is the illustration by a sequence of up and down arrows, from right to left, by putting a left-down arrow starting at the current vertex for an original arrow and a right-down arrow ending at the vertex for an inverse arrow.
\end{definition}

\begin{example}
In the following quiver $Q$:
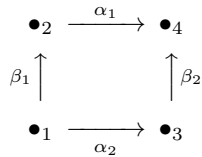
\begin{align*}
\begin{CD}
\bullet_2 @>\alpha_1>> \bullet_4\\
@A\beta_1 AA          @AA\beta_2 A\\
\bullet_1 @>>\alpha_2> \bullet_3
\end{CD}
\end{align*}
since $R_Q^3 = 0$, the zero ideal $I=0$ is admissible.  It is easy to see $A = kQ/I = kQ$ is a gentle algebra. Following the description above, the diagram of $C = \alpha_2\beta_1^{-1}\alpha_1^{-1}\beta_2$ in $\str(Q)$ is the following:

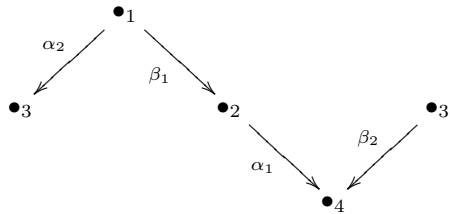
\begin{align*}
\xymatrix{
  & \bullet_1 \ar[dr]_{\beta_1} \ar[dl]_{\alpha_2}\\
\bullet_3 & & \bullet_2 \ar[dr]_{\alpha_1} &  &\bullet_3 \ar[dl]_{\beta_2}\\
&& & \bullet_4 
}
\end{align*}
\end{example}

\subsection*{String Modules}

Considering a string algebra $B$, to every string $C$ we associate an indecomposable $B$-module $M(C)$, constructed as follows: in the diagram of $C$, as defined in Definition \ref{stringdiagram}, put a one dimensional $k$-vector space at each vertex and the identity map for each arrow connecting two consecutive vertices of the diagram. In the representation $M(C)$, the dimension of the vector space associated to each vertex of the quiver is given by the number of times that the string $C$ passes through the aforementioned vertex. Consequently, the dimension vector of $M(C)$ has an explicit description. Provided it does not cause confusion, we do not distinguish between a string $C$ and the associated string module $M(C)$ and we use $C$ to refer to both. For details see \cite{CB, BR}. 

For a string $C$, by means of the equivalence relation introduced in \cite{BR}, we always identify $M(C)$ and $M(C^{-1})$, which are isomorphic as $B$-modules. 
Two string modules are isomorphic if and only if their strings are equivalent.

The indecomposable modules which are not string modules are associated to bands,
and are called band modules.  Since our focus is on string modules, we do
not discuss the details of band modules here.

As shown in \cite{BR}, the Auslander-Reiten translation of string modules over a string algebra admits an explicit combinatorial description.  In order to
state it, we will need to introduce some further combinatorial notions.

\begin{lemma}\label{uniqueadd}
  Let $W$ be a string of positive length, and let us pick an orientation
  $\epsilon$ where
  $\epsilon\in\{\textrm{direct},\textrm{inverse}\}$.  There is at most one
  way to add an arrow preceding $s(W)$ whose orientation agrees with $\epsilon$,
  such that the resulting path is still a string.  Similarly, there is at
  most one way to add such an arrow following $e(W)$.
%
\end{lemma}

\begin{proof} The lemma follows immediately from conditions (S1) and (S2).
\end{proof}

In these cases we say that $W$ can be extended at its start, or at its end, by an
inverse arrow, or a direct arrow, as applicable.  Note that the lemma is
not true as stated if $W$ is of length zero.  

If a string $W$ can be extended at its start by an inverse arrow, then we
consider the result of adding an inverse arrow at $s(W)$, followed by adding
as many direct arrows as possible.  We call this adding a cohook at
$s(W)$.  We denote the result of this operation $\Wc$.  It is well-defined
by Lemma \ref{uniqueadd}.

Symmetrically, if $W$ can be extended by a direct arrow at $e(W)$, then
we consider the result of adding a direct arrow at $e(W)$ followed by as
many inverse arrows as possible, and we call this adding a cohook at $e(W)$.
We denote the result of this by $\cW$.  

If it is possible to add a cohook at each end of $W$, we write $\cWc$ for
the result of doing so. 

If $W$ is of length zero, i.e., a single vertex, there may be two arrows
pointing towards $W$, and in that case, two cohooks can be added to $W$.
We define $\cWc$ to be the result of adding both cohooks.

We now describe the operation of removing a hook from $W$.  To remove
a hook from the start of $W$, we look for the first direct arrow in $W$,
and we remove it, together with all its preceding inverse arrows.  In other
words, we find a factorization $W= X \theta \mathcal I$, with
$\mathcal I$ an inverse string and $\theta$ a direct arrow.  The result of removing a hook from the beginning of
$W$ is $X$.  We write $X=\Wh$.  Note that this is not defined if $W$
contains no direct arrows.

Similarly, to remove a hook from the end of $W$, we look for the last inverse
arrow of $W$ and remove it, together with all subsequent direct arrows.
In other words, we find a factorization of $W$ as $W=\mathcal D \gamma^{-1}
Y$, with $\mathcal D$ direct; then the result of removing a hook from the end
of $W$ is $Y$.  We write $Y=\hW$.


\begin{theorem}\cite{BR} \label{tau-th}
  Let $B=kQ/I$ be a string algebra, and let $W$ be a string module.  At
  either end of $W$, if it is possible to add a cohook, add a cohook.
  Then, at the ends at which it was not possible to add a cohook, remove
  a hook.  The result is $\tau_B W$.
\end{theorem}

Note that if it is possible to add a cohook at exactly one endpoint of $W$,
then after having done so, it will be possible to remove a hook from the
other end.  If it is impossible to add a cohook at either end, then either it is possible to remove a hook from each end, or else the module was projective,
in which case we interpret the result of removing a hook from each end as
the zero module (which is consistent with the fact that the Auslander-Reiten
translation of a projective module is zero).

\subsection*{Morphisms between string modules}
A basis for the space of all homomorphisms between two strings was given by W. Crawley-Boevey in \cite{CB} in a more general setting. In \cite{S}, J. Schröer reformulated the aforementioned basis for string modules. In this paper, we mainly use Schröer's reformulation and notation for the description of $\Hom$-space. 

\begin{definition}
For $C \in \str(Q)$, the set of all factorizations of $C$ is defined as 
$$\mathcal{P}(C)=\{(F,E,D) \mid F,E,D \in \str(Q) \textsf{ and } C=FED\}.$$
Moreover, if $(F,E,D) \in \mathcal{P}(C)$, we write  $(F,E,D)^{-1}=(D^{-1},E^{-1},F^{-1})\in \mathcal{P}(C^{-1})$.

A triple $(F,E,D) \in \mathcal{P}(C)$ is a called a \emph{quotient factorization} of $C$ if the following hold:

\begin{enumerate}[(i)]			

\item $D=e_{s(E)}$ or $D=\gamma D'$ with $\gamma$ in $Q^{op}$;	

\item $F=e_{e(E)}$ or $F=F'\theta$ with $\theta$ in $Q$.

\end{enumerate}

A quotient factorization $(F,E,D)$ induces a surjective quotient map from
$C$ to $E$.
The set of all quotient factorizations of $C$ is denoted by $\mathcal{F}(C)$.

A factor of $C$ is generally of the following form, where
$D$ and $F$ can also be lazy paths. 

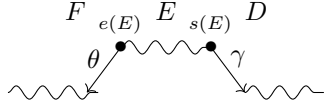
\begin{center}
\begin{tikzpicture}[scale=0.6]
\draw  [-,decorate,decoration=snake] (0,0) --(2,0);
----
\draw [->] (0,0) -- (-0.75,-1);
\node at (-0.6,-0.25) {$\theta$};
\draw [-,decorate,decoration=snake] (-2.5, -1) -- (-0.5,-1);
----
\draw [->] (2,0) -- (2.75,-1);
\node at (2.6,-0.25) {$\gamma$};
\draw [-,decorate,decoration=snake] (2.75,-1) -- (4.5,-1);
----
\node at (1,0.8) {$E$};
\node at (2,0.5) {$_{s(E)}$};
\node at (0,0.5) {$_{e(E)}$};
\node at (2,0.0) {$\bullet$};
\node at (0,0.0) {$\bullet$};
\node at (3,0.8) {$D$};
\node at (-1,0.8) {$F$};
\end{tikzpicture}
\end{center}
\bigskip

Dually, $(F,E,D) \in \mathcal{P}(C)$ is called a \emph{submodule factorization} of $C$ if the following hold:

\begin{enumerate}[(i)]

\item $D=e_{s(E)}$ or $D=\gamma D'$ with $\gamma$ in $Q$;

\item $F=e_{e(E)}$, or $F=F'\theta$ with $\theta$ in $Q^{op}$.
\end{enumerate}

The set of all submodule factorizations of $C$ is denoted by $\mathcal{S}(C)$.
A submodule factorization $(F,E,D)$ induces an inclusion of $E$ into $C$.

\end{definition}


\begin{definition}\label{admissible_pair}
For $C_1$ and $C_2$ in $\str(Q)$, a pair $((F_1,E_1,D_1),(F_2,E_2,D_2)) \in \mathcal{F}(C_1) \times \mathcal{S}(C_2)$ is called \emph{admissible} if $E_1 \sim E_2$ (i.e, $E_1=E_2$ or $E_1=E_2^{-1}$). The collection of all admissible  pairs is denoted by $\mathcal{A}(C_1,C_2)$.
\end{definition}

For each admissible pair $T=((F_1,E_1,D_1),(F_2,E_2,D_2))$ in $\mathcal{A}(C_1,C_2)$, there exists a homomorphism between the associated string modules $$f_T: C_1 \rightarrow C_2, $$ defined as the composition of the projection
from $C_1$ to $E_1$, followed by the identification of $E_1$ with $E_2$, followed by the inclusion of 
$E_2$ into $C_2$.  
We refer to these homomorphisms $f_T$ as {\em graph maps.}

The following theorem plays a crucial role in this paper:

\begin{theorem}\cite{CB}\label{Thm-graph-maps}
If $A=kQ/I$ is a string algebra and $C_1$ and $C_2$ are string modules, the set of graph maps $\{f_T\mid T\in \mathcal{A}(C_1,C_2) \}$ is a basis for $\Hom_A(C_1, C_2)$.
\end{theorem}

\subsection*{Auslander-Reiten duality}
Write $\underline{\Hom}_B(M,N) = \Hom_B(M,N)/P(M,N)$ for the quotient of $\Hom_B(M,N)$ by the morphisms factoring through a projective module.

Similarly, write $I(M,N)$ for the morphisms from $M$ to $N$ factoring through injectives, and
$\overline{\Hom}_B(M,N)$ for $\Hom_B(M,N)/I(M,N)$.


The following functorial isomorphisms, known as the \emph{Auslander-Reiten formulas}, play an important role in the rest of this paper:

\begin{theorem}\label{ARformula}
For every pair of $B$-modules $X$ and $Y$, we have the following:

\begin{itemize}
\item $\overline{\Hom}_B(X,\tau_B Y) \simeq D \Ext^1_B(Y,X)$;
\item $D\underline{\Hom}_B(Y,X) \simeq \Ext^1_B(X,\tau_B Y)$.
\end{itemize}
\end{theorem}

\subsection*{$\tau$-tilting theory}
$\tau$-tilting theory was recently introduced by Adachi, Iyama, and Reiten \cite{AIR}.  
An $A$-module module is said to be \emph{$\tau$-rigid} if $\Hom(X,\tau X)=0$.  
The module $X$ is said to be 
\emph{$\tau$-tilting} if it is $\tau$-rigid and it has as many non-isomorphic indecomposable summands as $A$.  The module $X$ is said to be \emph{support $\tau$-tilting} if $X$ is $\tau$-rigid and the number of indecomposable summands of $X$ equals the number of non-isomorphic simple modules in its composition series.  

A \emph{torsion class} in $\modu A$ is a full subcategory closed under quotients and 
extensions.  If $X$ is an $A$-module, we write $\Fac X$ for the full subcategory of $\modu A$ consisting of quotients of sums of copies of $A$.  
A torsion class is called {\it functorially finite} if
it can be written as $\Fac X$ for some $A$-module $X$.  

An Ext-projective module in a subcategory $T$ of $\modu A$ is a module $X$ such that 
$\Ext^1_A(X,-)$ vanishes on all modules in $T$.  

\begin{theorem}[{\cite[Theorem 2.7]{AIR}}]\label{st-tor}
There is a bijection from basic support $\tau$-tilting modules to 
functorially finite torsion classes, which sends the support $\tau$-tilting module 
$X$ to $\Fac X$, and sends the functorially finite torsion class $T$ to the direct sum of its $\Ext$-projective indecomposables.  
\end{theorem}

\section{Fringed Algebras}

\begin{definition} For a gentle algebra $A=kQ/I$, we call a vertex of $Q$ \emph{defective} if it does not have exactly two incoming arrows and two outgoing arrows.

\end{definition}

Our strategy is to define a larger quiver $\hat Q$ containing $Q$, and an
ideal $\hat I$ in $k\hat Q$ containing $I$, such that:
\begin{itemize}
  \item all the vertices of
    $Q$ will be non-defective in $\hat Q$, and
  \item $k\hat Q/\hat I$ is still a gentle algebra.
\end{itemize}
    We call this process
\emph{fringing}.  We begin by describing $\hat Q$.  

\begin{definition}
  For a gentle algebra $A=kQ/I$, let $\hat Q$ be the quiver obtained by
  adding up to two new vertices with arrows pointing towards $v$ and up to
  two new vertices with arrows from $v$, for each defective vertex $v$ of
  $Q$, such that in $\hat Q$, none of the original vertices of $Q$ are
  defective anymore.  We call $\hat Q$ the \emph{fringed quiver} of $Q$, and
  $\hat Q_0 \setminus Q_0$ the \emph{fringe vertices}.
\end{definition}
  
  We now define $\hat I$.  

\begin{def-prop}
  Let $A=kQ/I$ be a gentle algebra.  Let $\hat I$ be the ideal of $k\hat Q$
  which contains $I$, such that $\hat A=k\hat Q/\hat I$ is a gentle algebra.
  $\hat I$ is well-defined up to relabelling the vertices of $k\hat Q$.
  We call $\hat I$ the \emph{fringed ideal} and $\hat A$ the \emph{fringed algebra}.
  \end{def-prop}

The well-definedness of $\hat I$ is immediate from the fact that the
relations added to $\hat I$ consist of certain length two paths
through vertices that were defective in $Q$, and there is no choice about
which relations to add up to relabelling the fringe vertices.  
It is obvious that $A$ is a subalgebra of $\hat{A}$.  Also,
$A= \hat A/(e_F)$, where $e_F= \sum_{f\in \hat{Q}_0\setminus Q_0} e_f$, is the sum of all
the idempotents associated to the fringe vertices of $\hat{Q}$.
We refer to $e_F$ as the \emph{fringe idempotent} of $\hat{A}$.

\begin{example}

  Let $Q$ be the quiver below consisting of the black vertices and the arrows
  between them.  
  All the vertices of $Q$ are defective. The corresponding fringed quiver
  $\hat{Q}$ is obtained by adding the white vertices and arrows, with
  the result that every black vertex has exactly two incoming and two outgoing arrows.

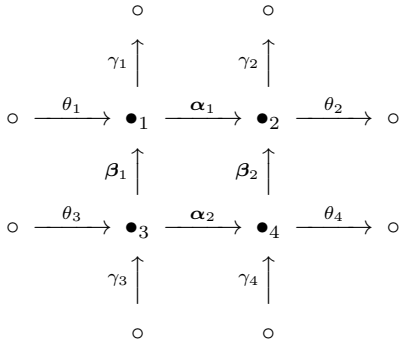
\begin{align*}
 \begin{CD}
    @. \circ @. \circ @. \\
    @. @A{\gamma_1}AA  @A{\gamma_2}AA @.\\
  \circ @>{\theta_1}>> \bullet_1 @>{\boldsymbol\alpha_1}>> \bullet_2 @>{\theta_2}>> @. \circ\\
  @.  @A{\boldsymbol\beta_1}AA  @A{\boldsymbol\beta_2}AA @. \\
  \circ @>{\theta_3}>> \bullet_3 @>{\boldsymbol\alpha_2}>> \bullet_4 @>{\theta_4}>> @. \circ\\
  @.  @A{\gamma_3}AA @A{\gamma_4}AA @.\\
   @. \circ @. \circ @. \\
 \end{CD}
\end{align*}\\

If we consider the admissible ideal $I$ in $kQ$ generated by $\alpha_1\beta_1$ and $\beta_2\alpha_2$, we have that $A= kQ/I$ is a gentle algebra. The corresponding ideal $\hat{I}$ in $k\hat Q$ is generated by $\{\beta_2\alpha_2, \alpha_1\beta_1, \gamma_2\alpha_1, \alpha_2\gamma_3, \theta_2\beta_2, \beta_1\theta_3, \theta_4\gamma_4, \gamma_1\theta_1\}$.  Note that we could have swapped the roles of, for example,
$\gamma_3$ and $\theta_3$ in $\hat I$, but that is equivalent to relabelling their sources.  The associated fringed algebra is $\hat{A}=k\hat{Q}/\hat{I}$.
\end{example} 

The previous example admits a natural generalization, which is closely related to the work of McConville \cite{McC}.

\begin{example}\label{gen-ex}
 Let $Q_{k,n}$ be the following quiver, which is an orientation of a $k$ by $n-k$ grid:
 \begin{align*}
 \begin{CD}
\bullet_{1,1} @>{\boldsymbol\alpha_{1,1}}>> \bullet_{1,2} @>{\boldsymbol\alpha_{1,2}}>> \cdots @>{\boldsymbol\alpha_{1,n-k-1}}>> \bullet_{1,n-k} \\
  @A{\boldsymbol\beta_{1,1}}AA  @A{\boldsymbol\beta_{1,2}}AA @. @A{\boldsymbol\beta_{1,n-k}}AA \\
 \bullet_{2,1} @>{\boldsymbol\alpha_{2,1}}>> \bullet_{2,2} @>{\boldsymbol\alpha_{2,2}}>> \cdots @>{\boldsymbol\alpha_{2,n-k-1}}>> \bullet_{2,n-k} \\
 @A{\boldsymbol\beta_{2,1}}AA  @A{\boldsymbol\beta_{2,2}}AA @. @A{\boldsymbol\beta_{2,n-k}}AA \\
 \vdots @. \vdots @. \ddots @. \vdots\\
  @A{\boldsymbol\beta_{k-1,1}}AA  @A{\boldsymbol\beta_{k-1,2}}AA @. @A{\boldsymbol\beta_{k-1,n-k}}AA \\
 \bullet_{k,1} @>{\boldsymbol\alpha_{k,1}}>> \bullet_{k,2} @>{\boldsymbol\alpha_{k,2}}>> \cdots @>{\boldsymbol\alpha_{k,n-k-1}}>> \bullet_{k,n-k} \\ 
 \end{CD}
\end{align*}\\
Consider the ideal $I$ generated by all relations of the form $\alpha\beta$ and $\beta\alpha$.

The fringed quiver $\hat{Q}_{k,n}$ is obtained by adding the white vertices
and incident arrows, so that none of the black vertices is defective:
 \begin{align*}
 \begin{CD}
@. \circ @. \circ @. \cdots @. \circ \\
@. @A{\beta_{0,1}}AA  @A{\beta_{0,2}}AA @. @A{\beta_{0,n-k}}AA \\
\circ @>{\alpha_{1,0}}>> \bullet_{1,1} @>{\boldsymbol\alpha_{1,1}}>> \bullet_{1,2} @>{\boldsymbol\alpha_{1,2}}>> \cdots @>{\boldsymbol\alpha_{1,n-k-1}}>> \bullet_{1,n-k} @>{\alpha_{1,n-k}}>> \circ \\
@.  @A{\boldsymbol\beta_{1,1}}AA  @A{\boldsymbol\beta_{1,2}}AA @. @A{\boldsymbol\beta_{1,n-k}}AA \\
\circ @>{\alpha_{2,0}}>>  \bullet_{2,1} @>{\boldsymbol\alpha_{2,1}}>> \bullet_{2,2} @>{\boldsymbol\alpha_{2,2}}>> \cdots @>{\boldsymbol\alpha_{2,n-k-1}}>> \bullet_{2,n-k} @>{\alpha_{2,n-k}}>> \circ \\
@. @A{\boldsymbol\beta_{2,1}}AA  @A{\boldsymbol\beta_{2,2}}AA @. @A{\boldsymbol\beta_{2,n-k}}AA \\
 \vdots @. \vdots @. \vdots @. \ddots @. \vdots @. \vdots \\
@.  @A{\boldsymbol\beta_{k-1,1}}AA  @A{\boldsymbol\beta_{k-1,2}}AA @. @A{\boldsymbol\beta_{k-1,n-k}}AA \\
\circ @>{\alpha_{k,0}}>>  \bullet_{k,1} @>{\boldsymbol\alpha_{k,1}}>> \bullet_{k,2} @>{\boldsymbol\alpha_{k,2}}>> \cdots @>{\boldsymbol\alpha_{k,n-k-1}}>> \bullet_{k,n-k} @>{\alpha_{k,n-k}}>> \circ \\ 
@. @A{\beta_{k,1}}AA  @A{\beta_{k,2}}AA @. @A{\beta_{k,n-k}}AA \\
@. \circ @. \circ @. \cdots @. \circ @. @. @. .\\
 \end{CD}
\end{align*}\\
The new ideal $\hat{I}$ is generated by all relations of the form $\alpha\beta$ and $\beta\alpha$, including the new arrows. 
\end{example}

  This class of examples is sufficiently complicated to illustrate most of the
  phenomena which will be of interest to us in this paper.  We will therefore
  sometimes draw examples of string modules in a shape as in Example \ref{gen-ex}, suppressing the underlying grid-shaped quiver, and drawing only the strings.  

\begin{definition} For a string $X$ in $Q$, we write $\cohook(X)$ for the result of adding cohooks to both ends of $X$ in the fringed quiver $\hat{Q}$. We
  call this the cohook completion of $X$.  
\end{definition}
  
The diagram of $\cohook(X)$ is illustrated as follows:

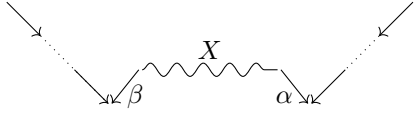
\begin{center}
\begin{tikzpicture}[scale=0.9]
\draw  [-,decorate,decoration=snake] (0,0) --(2,0);
\node at (1,0.3) {$X$};
\draw [->] (-0.05,0) -- (-0.45,-0.5);
\node at (-0.1,-0.4) {$\beta$};
\draw [->] (-1,0) -- (-0.5,-0.5);
\draw [dotted] (-1.5, 0.5) -- (-0.5,-0.5);
\draw [->] (-2,1) -- (-1.5, 0.5);
\node at (4,0.3) {};
----
\draw [->] (2.05,0) -- (2.45,-0.5);
\node at (2.1,-0.4) {$\alpha$};
\draw [->] (3,0) -- (2.5,-0.5);
\draw  [dotted] (3.5,0.5) -- (3,0);
\draw [->] (4,1) -- (3.5,0.5);
\node at (-2,0.3) {};
\end{tikzpicture}
\end{center}

Here we refer to the arrows $\alpha$ and $\beta$ adjacent to $X$ as the \emph{shoulders} of $\cohook(X)$ and the sequence of direct arrows on the right and the sequence of inverse arrows on the left as the \emph{arms} of $\cohook(X)$. 

In the following proposition we give the description of $\tau_{\hat A}(X)$ for $X$ a string in mod ${A}$.

\begin{proposition}\label{tau-cohook} For a pair of strings $X$ and $Y$ in $A$, we have the following:
\begin{enumerate}
\item $\tau_{\hat A}Y= \cohook(Y)$;
\item  $\Ext^1 _{A}(Y,X) \simeq \Ext^1 _{\hat{A}}(Y,X)$.
\end{enumerate}
\end{proposition}

\begin{proof}
\begin{enumerate}
\item Because $Y$ is an $A$-module, it is possible to add
  cohooks to both ends of it when we think of it inside $\hat Q$.  Therefore, Theorem \ref{tau-th} tells us
  that $\tau_{\hat A}Y$ is obtained by adding cohooks to both ends of $Y$.

\item Recall that $e_F$ is the fringe idempotent in $\hat{A}$, and that
  $A= \hat{A}/(e_F)$. Hence,  $\Ext^1 _{A}(Y,X) \simeq \Ext^1 _{\hat{A}}(Y,X)$.
\end{enumerate}
\end{proof}

The following lemma relates $\tau_A Y$ and $\tau_{\hat A}Y$.  

\begin{lemma} \label{sub-sub} Let $Y$ be a string for $Q$.  
Then $\tau_{A}Y$ is a submodule of $\tau_{\hat A}Y$.  
\end{lemma}

In fact, it is known that if $B$ is an algebra, $B/I$ is a quotient algebra,
and $M$ is a $B/I$-module, then $\tau_{B/I}M$ is a submodule of 
$\tau_B M$, see \cite[Lemma VIII.5.2]{ASS}.  We include
a proof of the special case we state because it is an easy consequence
of the combinatorics of strings.  

\begin{proof} This follows from the fact that $\tau_{\hat A}Y$ is
defined by adding cohooks in $\hat Q$ to $Y$, while $\tau_AY$ is defined
by, at each end, either adding a cohook in $Q$ (which will coincide with the
cohook in $Y$ except that it will be missing the final arrow) or subtracting
a hook; the result is that $\tau_AY$ is a substring of $\tau_{\hat A}Y$ and
at each end, the first arrow missing from $\tau_AY$ points towards 
$\tau_AY$.  Thus, $\tau_AY$ is a submodule of $\tau_{\hat A}Y$.  
\end{proof}

\section{Kisses}\label{kiss-admissible}

In this section we give an interpretation of $\Hom_{\hat{A}}(X,\tau_{\hat{A}}Y)$ for a pair of  strings $X$ and $Y$ in $A$. 
Our technique is inspired by work of Schröer \cite{S} and the combinatorial framework developed by McConville \cite{McC}.  


In order to apply Theorem \ref{tau-th} to analyze $\Hom_{\hat A}(X,\tau_{\hat A} Y)$,
we need to introduce a new notion and fix some notation, as follows:

\begin{definition} For a pair of strings $X$ and $Y$, with factorizations $X=(X'',Z,X')$ and $Y=(Y'',Z,Y')$, we say there exists a \emph{kiss from $X$ to $Y$ along $Z$}, provided:
\begin{enumerate}[(i)]
\item $(X'',Z,X')$ is a quotient factorization,
\item $(Y'',Z,Y')$ is a submodule factorization,
  \item all of $X'',X',Y'',Y'$ have strictly positive length.
\end{enumerate}
Such a kiss is denoted by $\keg$.
\end{definition}

We emphasize that the notion of kiss is directed. A kiss from $X$ to $Y$ can be illustrated as follows.  Note that
the four arrows $\gamma$, $\zeta$, $\theta$, and $\sigma$ 
must all appear, and must be oriented as shown.  

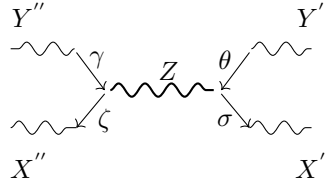
\begin{center}
\begin{tikzpicture}[scale=0.50]
\draw  [-,decorate,decoration=snake][thick] (0.15,0) --(2.9,0);
\node at (1.7,0.5) {$Z$};
------
\draw [<-] (0,0) -- (-0.75,1);
\draw [-,decorate,decoration=snake] (-0.8,1.1) -- (-2.5,1.1);
\node at (-0.2,0.8) {$\gamma$};
\node at (-2,2) {$Y^{''} $};
\node at (5.5,2) {$Y^{'} $};
-----
\draw [<-] (3.1,0) -- (3.85,1);
\draw [-,decorate,decoration=snake] (3.9,1.1) -- (5.5,1.1);
\node at (3.2,0.8) {$\theta$};

------
\draw [<-] (-0.75,-1) -- (0,-0.1);
\draw [-,decorate,decoration=snake] (-2.5,-1) -- (-0.75,-1);
\node at (0,-0.8) {$\zeta$};
\node at (-2,-2) {$X^{''}$};
\node at (5.5,-2) {$X^{'}$};
------
\draw [<-] (3.85,-1) -- (3.1,-0.1);
\draw [-,decorate,decoration=snake] (3.85,-1) -- (5.5,-1);
\node at (3.2,-0.8) {$\sigma$};

\end{tikzpicture}
\end{center}

By $\kiss(X, Y)$ we denote the number of kisses from $X$ to $Y$, whereas we use $\Kiss(X, Y)$ for the set of all kisses from $X$ to $Y$, thought of as a set of pairs of triples, as above.
This is a generalization of a notion introduced by McConville \cite{McC}.

\begin{lemma} \label{kiss-a-module}
Let $X$ and $Y$ be strings in $Q$.  
Let $\keg$ be a kiss from $\cohook(X)$ to $\cohook(Y)$. Then \begin{enumerate}
\item $Z$ is a quotient of $X$.  
\item $Z$ is a submodule of $\tau_A Y$.
\end{enumerate}
\end{lemma}

Note that $Z$ is by definition a factor of $\cohook(X)$; the point of (1) is that
it is really a quotient of $X$ (though not necessarily a proper quotient:
it might be equal to $X$).  Note also that $Z$ is by definition a submodule of $\tau_{\hat A}Y$, but, as depicted in the following example, $Z$ is not necessarily a submodule of $Y$. In this example, the cohooks are drawn dashed. Note that $\alpha$ and $\beta$ belong to the kiss from $\cohook(X)$ to $\cohook(Y)$, but they are not in $Y$:

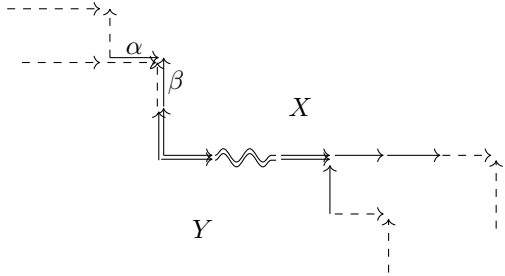
\begin{center}
\begin{tikzpicture}[scale=0.65]

\draw  [->] (0.1,0) --(1.1,0);
\draw  [->] (1.2,-1) --(1.2,0);
\draw  [->] (1.2,-2) --(1.2,-1.03);
\draw [->] (1.2,-2) -- (2.2,-2);
\draw [-,decorate,decoration=snake] (2.25,-2) -- (3.5,-2);
\draw [->] (3.6,-2) -- (4.6,-2);
\draw [->] (4.7,-2) -- (5.7,-2);
\draw [->] (5.77,-2) -- (6.86,-2);

\node at (4,-1) {$X$};
-----
\draw  [->,dashed] (0.1,0) --(0.1,1);
\draw  [->,dashed] (-2,1) --(-0.1,1);

\draw [->,dashed] (6.9,-2) -- (7.9,-2);
\draw [->,dashed] (8,-3.5) -- (8,-2.1);

------
\draw  [->] (1.1,-2.1) --(1.1,-1.1);
\draw [->] (1.15,-2.08) -- (2.2,-2.08);
\draw [-,decorate,decoration=snake] (2.25,-2.1) -- (3.5,-2.1);
\draw [->] (3.6,-2.08) -- (4.6,-2.08);
\draw [->] (4.6,-3.2) -- (4.6,-2.2);


\node at (2,-3.5) {$Y$};

-----
\draw  [->,dashed] (1.07,-1) --(1.07,-0.1);
\draw  [->,dashed] (0.05,-0.1) --(1.05,-0.1);
\draw  [->,dashed] (-1.7,-0.1) --(-0.1,-0.1);

\draw [->,dashed] (4.7,-3.2) -- (5.7,-3.2);
\draw [->,dashed] (5.8,-4.4) -- (5.8,-3.3);

\node at (0.6,0.2) {$\alpha$};
\node at (1.45,-0.5) {$\beta$};

\end{tikzpicture}
\end{center}

Note that in this example we are following the convention of considering our strings to lie in a grid quiver as in Example \ref{gen-ex}, without drawing the underlying grid.

\begin{proof} (1) If the start of $Z$ were before the start of $X$ in 
  $\cohook(X)$ then $X'$ would consist only of direct arrows, which is contrary
  to the definition of a kiss.  A similar argument shows that the end of
  $Z$ cannot be later than the end of $X$, so $Z$ is a substring of $X$.
  Because of the direction of the arrows just outside $Z$ in $\cohook(X)$
  (which are part of the definition of a kiss), $Z$ is a quotient of $X$.

  (2) Recall from Theorem \ref{tau-th} that $\tau_AY$ is obtained from $Y$ by
  adding cohooks if possible, and otherwise removing hooks.  

  Suppose that it was possible to add a cohook to the start of $Y$ in
  $A$.  In this case, both $\tau_AY$ and $\tau_{\hat A}Y$ add a cohook at the
  start of $Y$.  The only difference is that $\tau_{\hat A}Y$ will include
  one additional inverse arrow from a fringe vertex.  By the definition of
  a kiss, $Z$ cannot include that additional arrrow.  

  Now suppose that it was not possible to add a cohook at the start of $Y$.
  In this case, $\tau_AY$ was obtained by removing up to the first direct
  arrow of $Y$.  By the definition of a kiss, there is a direct arrow before
  $Z$, so what was removed does not intersect $Z$, and thus the start of
  $Z$ is not before the start of $\tau_AY$.  

  We have therefore established that the start of $Z$ is after the start of
  $\tau_AY$.  Similarly, the end of $Z$ is before the end of $\tau_AY$.
  Because of the directions of the arrows just outside $Z$ in $\tau_AY$,
  $Z$ is a submodule of $\tau_AY$.  
\end{proof}

By the previous lemma, a kiss from $\cohook(X)$ to $\cohook(Y)$ determines a 
non-zero morphism from $X$ to $\tau_AY$ and from $X$ to $\tau_{\hat A}Y$.  

\begin{theorem}\label{basis}

\begin{enumerate} \item The elements of $\Hom(X,\tau_{\hat A}Y)$ corresponding to kisses from $\cohook(X)$
to $\cohook(Y)$ form a basis.  
\item The elements of $\Hom(X,\tau_AY)$ corresponding to kisses from
$\cohook(X)$ to $\cohook(Y)$ form a basis.
\end{enumerate}
\end{theorem}

\begin{proof} By the description of morphisms between string modules in
terms of graph maps, Lemma \ref{kiss-a-module} implies that the kisses
define a linearly independent collection inside each of the Hom spaces, and also span the Hom spaces. 
\end{proof}

\begin{theorem}\label{dims} Let $X$ and $Y$ be strings in $Q$.  
$$\kiss(\cohook(X),\cohook(Y))=\dim \Hom_{\hat A}(X,\tau_{\hat A}Y)=\dim
\Hom_A(X,\tau_AY).$$
\end{theorem}

\begin{proof} This theorem is immediate from the stronger Theorem
\ref{basis} above.
\end{proof}

\begin{proposition}\label{Hom-isomorphism} Let $X$ and $Y$ be strings in $Q$.  
The canonical map from $\tau_{A}Y$ to $\tau_{\hat A}Y$ 
induced by the fact that the former is a submodule of the latter,
induces an isomorphism from $\Hom(X,\tau_{A}Y)$ to $\Hom(X,\tau_{\hat A} Y)$.  
\end{proposition}

\begin{proof} This is immediate from Theorem \ref{basis} and Lemma
\ref{sub-sub}.\end{proof}

\section{Maximal non-kissing collections}~\label{maximal-non-kissing}

In this section, we show that support $\tau$-tilting modules for a gentle algebra correspond to certain collections of strings in its fringed quiver which can be combinatorially characterized.  

Let $B$ be an algebra with $n$ simples.  
We consider the set of $\tau$-rigid indecomposable
modules together with $n$ formal objects $P_i[1]$, on which we define a 
relation of compatibility.  Two $B$-modules $M$ and $N$ are compatible if 
$\Hom(M,\tau N)=0=\Hom(N,\tau M)$.  A $B$-module $M$ and $P_i[1]$ are
compatible if $\Hom(P_i,M)=0$.  $P_i[1]$ and $P_j[1]$ are always compatible.
There is a bijective correspondence from maximal compatible collections
to basic support $\tau$-tilting modules (see \cite{AIR}): one takes the sum of the modules 
in the collection, and throws away whatever $P_i[1]$ appear.

Let $A=kQ/I$ be a gentle algebra, and let $\hat A=k\hat Q/I$ be its fringed algebra.
For $v$ a vertex of $\hat Q$, define the \emph{injective string} of $v$ to be the 
string obtained by adding to the lazy path at $v$ both maximal sequences of arrows oriented towards $v$ in $\hat Q$.  We denote it $I_v$.  The corresponding $\hat A$-module is an
indecomposable injective module.

We define the set of \emph{long strings} of $\hat Q$ to consist of 
$\{\cohook(X) \mid X\in \str(Q)\} \cup \{I_v \mid v\in Q_0\}$.  
We call these long strings because they run between two fringe vertices,
so they are in a sense maximally long.  

\begin{theorem}~\label{thm-maximal-non-kissing} There is a bijective correspondence 
from 
maximal 
compatible collections of $A$-modules to 
maximal non-kissing collections of long strings of $\hat Q$, induced by the correspondence:
\begin{eqnarray*}
X \in \str(Q) &\to& \cohook(X)\\
P_v[1] \textrm{ for } v\in Q_0 &\to& I_v.
\end{eqnarray*}
\end{theorem}
 
\begin{proof} We have to check that the bijection above takes the 
compatibility relation to the non-kissing relation.  

Two $A$-modules are compatible if and only if the 
corresponding long strings do not kiss, by Theorem \ref{dims}.  

Any two injective strings are compatible, which is consistent with the
fact that $P_i[1]$ and $P_j[1]$ are compatible.  

An $A$-module
$M$ is compatible with $P_v[1]$ if and only if $\Hom_A(P_v,M)=0$, if and
only if $M$ does not pass through vertex $v$.  It suffices to show that 
there is a kiss from $\cohook(M)$ to $I_v$ if and only if $M$ passes through $v$,
and there is never a kiss from $I_v$ to $\cohook(M)$.  We now verify this.  

If $M$
passes through $v$, then in each direction from $v$, 
$\cohook(M)$ eventually leaves
the injective string from $v$ (at the endpoint of $M$ if not before), and
it leaves it along an arrow pointing away from the injective string.  Thus
there is a kiss from $\cohook(M)$ to $I_v$.  

If $M$ does not pass through $v$, then any common substring of $M$ and $I_v$
must consist of arrows all oriented in the same direction, with the arrow of
$I_v$ on either side of the substring also having the same orientation.  
Since this means that the arrows of $I_v$ on either side of the common
substring point in opposite directions, it cannot be a kiss.  The same argument applies to a common substring of an arm of $\cohook(M)$ and $I_v$, unless $v$
lies on an arm of $\cohook(M)$.  But in that case, there is a direction in which $\cohook(M)$ and $I_v$ never separate, so this is not a kiss either.  

Because all the arrows of $I_v$ are oriented towards $v$ in the middle of
the string, it is impossible for there to be a kiss from $I_v$ to any string.   
\end{proof}

\section{Poset of functorially finite torsion classes}\label{torposet}

It is natural to order the functorially finite torsion classes by inclusion.  Since by Theorem \ref{st-tor} there is a bijection between functorially finite torsion classes and support $\tau$-tilting modules, this can also be thought of as a poset structure 
on support $\tau$-tilting modules.  Having shown in the previous section that we can give a combinatorial description of the support $\tau$-tilting modules as maximal non-kissing collections, we proceed in this section to interpret this poset structure on maximal non-kissing collections.  

The following theorem combines a few different results from \cite{AIR}:
Theorem 2.18, the discussion following Definition 2.19, and Corollary
2.34.  

\begin{theorem}[{\cite{AIR}}] \label{agreement}
If $\mathcal T$ and $\mathcal U$ are two functorially finite
torsion classes with $\mathcal T$
properly contained in $\mathcal U$, and with no functorially finite
torsion class properly
between them, then their corresponding maximal compatible collections
are $\mathcal S \cup \{P\}$, $\mathcal S \cup \{R\}$, 
and conversely, given two 
maximal compatible collections $\mathcal S\cup \{P\}$ and 
$\mathcal S\cup\{R\}$, 
they correspond to functorially
finite torsion classes
which form a cover in the poset of torsion classes.
\end{theorem}

Suppose we have two maximal nonkissing collections of the form 
$\mathcal S \cup \{\cohook(X)\}$ and $\mathcal S \cup \{\cohook(Y)\}$.  Clearly
$\cohook(X)$ and $\cohook(Y)$ kiss; otherwise, this would violate
maximality of the collections.  By Theorem \ref{agreement}, these two
collections 
correspond to a pair of torsion classes which form a cover.  Suppose that
the torsion class corresponding to $\mathcal S \cup \{\cohook(X)\}$
contains the torsion class corresponding to $\mathcal S \cup \{\cohook(Y)\}$.
It follows that $Y$ is also in the first torsion class, while $X$ is 
$\Ext$-projective in the torsion class, so 
$\Hom(Y,\tau_A X)=0$ by \cite[Proposition 5.8]{AS}.  
The kiss(es) between $\cohook(X)$ and $\cohook(Y)$ are therefore from 
$\cohook(X)$ to $\cohook(Y)$. Similarly, if we have two maximal nonkissing collections
$\mathcal S \cup \{\cohook(X)\}$ and $\mathcal S \cup \{I_v\}$, then we see that there is a kiss from $\cohook(X)$ to $I_v$.  
This 
establishes the following theorem.

\begin{theorem}\label{covers} The cover relations in the poset of functorially finite torsion classes of $A$ can be described
in terms of their maximal non-kissing collections as follows: they correspond to pairs of maximal non-kissing collections of the form
$\mathcal S \cup \{C\}, \mathcal S \cup \{D\}$, and 
$\mathcal S\cup \{C\} > \mathcal S\cup \{D\}$ if the kisses go from
$C$ to $D$.  
\end{theorem}

An infinite poset is not necessarily characterized by its cover relations, but a finite poset is.  We therefore have the following corollary:

\begin{corollary} If $A$ has only finitely many functorially finite torsion classes, then the poset of torsion classes is isomorphic to the poset of maximal non-kissing collections, ordered by the transitive closure of $\mathcal S\cup\{C\} > \mathcal S\cup \{D\}$ where the kisses go from $C$ to $D$.  \end{corollary}

\section{Combinatorics of torsion classes}
\label{comb}

In this section we consider how the bijection between functorially finite torsion classes and support $\tau$-tilting module plays out for gentle algebras in terms of the combinatorics we have been developing.

We begin with the following theorem, describing the strings in the category of quotients of copies of a collection of strings.  

\begin{theorem} Let $A=kQ/I$ be a gentle algebra.  Let $X=\bigoplus X_i$ with each $X_i$ a string.  A string $Y$ is in $\Fac X$ if and only if $Y$ can be written as a union of strings each of which is a substring of $Y$ and
a factor of some $X_i$.  \end{theorem}

Here, when we write that $Y$ is a union of a certain set of strings, we mean that the strings may overlap, but each arrow of $Y$ occurs in 
at least one of the strings.  

\begin{proof} 
$Y$ is in $\Fac X$ if and only if there is a quotient map from a sum of copies of 
$X$ onto $Y$.  This means that, for each arrow of $Y$, we have to be able
to find some $X_i$ which maps to $Y$ and hits that arrow.  The map from
$X_i$ to $Y$ corresponds to a quotient of $X_i$ and a submodule of $Y$. 
\end{proof}






Let $\mathcal S$ be a collection of strings in $Q$, and let 
$\alpha$ be an arrow of $\hat Q$.  We will define $\Mc(\mathcal S,\alpha)$ to be a certain long string of $\hat Q$.  
We will construct it arrow by arrow in both directions from $\alpha$.  (We use the symbol $\Mc$ for this map to emphasize that this is the algebraic version of a map defined by McConville \cite[Section 8]{McC}.)

Let $\Gamma_0$ be the lazy path at $e(\alpha)$, and $\gamma_0=\alpha$.
We will define a sequence of
arrows $\gamma_i$, and strings $\Gamma_i=\gamma_i\Gamma_{i-1}$, for $i=1,2,\dots$.  (Note that we do not include $\alpha$ in $\Gamma_0$.)
Suppose we have already constructed $\gamma_1,\dots,\gamma_i$, and let
$u=e(\gamma_i)=e(\Gamma_i)$.  
If $u$ is a fringe vertex, we set $\im=i$ and stop.  Otherwise, we divide into cases:
\begin{itemize}
\item If $\gamma_i$ is a direct arrow and $\Gamma_i \in \mathcal S$, define $\gamma_{i+1}$ to be the unique direct arrow
such that $\gamma_{i+1}\gamma_i \not\in I$.  
\item If $\gamma_i$ is a direct arrow and $\Gamma_i\not \in 
\mathcal S$, define $\gamma_{i+1}$ to be the unique inverse arrow starting
at $u$ which is not $\gamma^{-1}_i$.  
\item If $\gamma_i$ is an inverse arrow and $\Gamma_i\in\mathcal S$, 
define $\gamma_{i+1}$ to be the unique direct arrow from $u$ other than $\gamma^{-1}_i$.  
\item If $\gamma_i$ is an inverse arrow and $\Gamma_i\not\in\mathcal S$, define $\gamma_{i+1}$ to be the unique inverse arrow such that $\gamma_i^{-1}\gamma_{i+1}^{-1}\not\in I$.
\end{itemize}
We continue in this way until we reach a fringe vertex.

We now extend the string in the opposite direction.  Define
$\Theta_0$ to be the lazy path at $s(\alpha)$.  We will proceed to define
a sequence of arrows $\gamma_{-j}$ and strings $\Theta_j=\Theta_{j-1}\gamma_{-j}$
for $j=1,2,\dots$.  Suppose that
we have already constructed $\gamma_{-1},\gamma_{-2}\dots\gamma_{-j}$, and
let $v=s(\gamma_{-j})$.  
If $v$ is a fringe
vertex, we set $\jm=j$ and stop.  Otherwise, we
define $\gamma_{-j-1}$ using the previous rule, but reversing the roles of direct and inverse arrows throughout.  Explicitly, we divide into cases as follows:
\begin{itemize}
\item If $\gamma_{-j}$ is an inverse arrow and $\Theta_j \in \mathcal S$, define $\gamma_{-j-1}$ to be the unique inverse arrow
such that $\gamma^{-1}_{-j-1}\gamma^{-1}_{-j} \not\in I$.  
\item If $\gamma_{-j}$ is an inverse arrow and $\Theta_j\not \in 
\mathcal S$, define $\gamma_{-j-1}$ to be the unique direct arrow starting 
at $v$ which is not the inverse of $\gamma_{-j}$.  
\item If $\gamma_{-j}$ is a direct arrow and $\Theta_j\in\mathcal S$, 
define $\gamma_{-j-1}$ to be the unique inverse arrow ending at $v$ other than $\gamma^{-1}_{-j}$.  
\item If $\gamma_{-j}$ is a direct arrow and $\Theta_j\not\in\mathcal S$, define $\gamma_{-j-1}$ to be the unique direct arrow such that $\gamma_{-j}\gamma_{-j-1}\not\in I$.
\end{itemize}
Now define $\Mc(S,\alpha)$ to be the concatenation of $\Gamma_{\im}$, $\alpha$,
and $\Theta_{\jm}$.  


\begin{theorem} \label{cang} Let $T$ be a functorially finite torsion class in $\modu A$.  Let
$\mathcal S_T$ be the strings in $T$.  
The collection of modules $\Mc(\mathcal S_T,\alpha)$, as 
$\alpha$ runs through the arrows of $\hat Q$, yields:
\begin{itemize}
	\item $\cohook(M)$ for $M$ an $\Ext$-projective of $T$ (each appearing for two choices of arrow $\alpha$),
    \item the injective string at
    each vertex of $Q$ over which no module in $T$ is supported (each appearing
    twice), and
    \item the injective string at each fringe vertex which is a sink
    (each appearing once).  
\end{itemize}
\end{theorem}

We divide the proof of the theorem into the next four propositions.  

  Let $M$ be a string in $Q$.  Let $\cohook(M)=\gamma_r \dots \gamma_0$, and let $M=\gamma_b\dots \gamma_a$ for some $0<a\leq b<r$.  Define
  $$X(M)=\{i \mid b+1\geq i\geq 1, \gamma_i \textrm{ is direct}, \gamma_b\dots\gamma_{i+1} \in \mathcal S_T\}.$$   
  We consider the condition $\gamma_b\dots\gamma_{i+1}\in \mathcal S_T$ to be vacuous when $i=b+1$, and $\gamma_{b+1}$ is a direct arrow since it is a
  shoulder of $\cohook(M)$, so $b+1\in X(M)$.  
  For $i=b$, the condition that $\gamma_b\dots\gamma_{i+1}\in \mathcal S_T$
  is interpreted as meaning that the lazy path at $e(\gamma_b)$ is in
  $\mathcal S_T$.  

  \begin{proposition}\label{module} Let $M$ be $\Ext$-projective in $T$,
    with $\cohook(M)=\gamma_r\dots\gamma_0$.  Let $x$ be the
    minimum element of $X(M)$.  Then $\Mc(\mathcal S_T,\gamma_x)=\cohook(M)$.
    \end{proposition}



  \begin{proof}
    Let us write $\Omega$ for the string $\gamma_b\dots\gamma_{x+1}$ which is,
    by assumption, in $\mathcal S_T$.  (If $x=b+1$, then $\Omega$ is not defined.)
    
  The proof is by induction.  Suppose that, in the construction of
  $\Mc(\mathcal S_T,\gamma_x)$, we have constructed $\Gamma_i$ and it is
  a substring of $\cohook(M)$, say $\Gamma_i=\gamma_y\dots\gamma_{x+1}$.  Suppose
  now that $\gamma_{y+1}$ is direct.  By our choice of $x$,
  $\gamma_b\dots\gamma_{x+1}\in \mathcal S_T$, and since $\gamma_{y+1}$ is
  direct, $\Gamma_i$ is a quotient substring of
  $\gamma_b\dots\gamma_{x+1}$, so it is also in $\mathcal S_T$.  Thus our
  algorithm chooses a direct arrow, necessarily $\gamma_{y+1}$.

  Suppose next that $\gamma_{y+1}$ is inverse.  
There is
a kiss from $\cohook(\Gamma_i)$ to $\cohook(M)$.  This implies that $\Gamma_i$
cannot be in $\mathcal S_T$.  Thus, our algorithm chooses an inverse arrow,
necessarily $\gamma_{y+1}$.

It follows by induction that $\Gamma_{i_{\rm max}}$ agrees with $\gamma_r\dots\gamma_{x+1}$, i.e., the part of $\cohook(M)$ after $\gamma_x$.

Now let us consider the part of $\cohook(M)$ before $\gamma_x$.  Suppose that
we have already constructed $\Theta_j$, and it is a substring of
$\cohook(M)$, say $\Theta_j=\gamma_{x-1}\dots\gamma_{z}$.  Suppose now that
$\gamma_{z-1}$ is direct.  If $\Theta_j$ were in $\mathcal S_T$, then
since, by assumption, $\Omega$ is in $\mathcal S_T$, so is their extension,
which contradicts the minimality of $x$. (If $x=b+1$, so $\Omega$ is not
defined, then $\Theta_j\in\mathcal S_T$ and $\gamma_{z-1}$ direct contradicts
the minimality of $x$.)  Thus, our algorithm chooses a direct arrow, necessarily $\gamma_{z-1}$.  

Suppose that $\gamma_{z-1}$ is inverse.  Then $\Theta_j$ is a quotient of
$M$, so it is in $\mathcal S_T$.  Thus our algorithm chooses an inverse arrow,
necessarily $\gamma_{z-1}$.  It follows by induction that $\Theta_{j_{\rm max}}$
agrees with $\gamma_{x-1}\dots\gamma_0$, i.e., the part of $\cohook(M)$ before
$\gamma_x$.  This completes the proof.  \end{proof}

  The above proposition shows that we can construct any Ext-projective of
  $T$ in two ways: once as in the proposition, and once applying
  the proposition to the reverse string.  Note that the orientations of
  the two chosen edges will be opposite, so this does indeed yield two
  distinct arrows $\alpha_1,\alpha_2$ of $\hat Q$ such that $M=\Mc(\mathcal S_T,\alpha_1)=\Mc(\mathcal S_T,\alpha_2)$.  

  We now consider the simpler case of injective strings.

  \begin{proposition} \label{inj} Let $T$ be a torsion class in $\modu A$.
    Let $v$ be a vertex such that no module of $T$ is supported at $v$,
    and let $I_v=\gamma_r\dots \gamma_0$
    with $\gamma_j$ direct for $i\geq j \geq 0$, and $\gamma_j$ inverse
    for $r\geq j \geq i+1$.  Let $x$ be minimal such that
    $\gamma_{i-1}\dots\gamma_{x+1}\in \mathcal S_T$.  (Again, this condition is vacuous for $x=i$.) Then
    $\Mc(\mathcal S_T,\gamma_x)=I_v$.  
  \end{proposition}

  \begin{proof} The proof is essentially the same as for Proposition
    \ref{module}.  

    Let us write $\Omega$ for the string $\gamma_{i-1}\dots\gamma_{x+1}$ which is,
    by assumption, in $\mathcal S_T$.  (If $x=i+1$, then $\Omega$ is not defined.)

    Since $\Omega\in \mathcal S_T$, it also follows that
    $\gamma_{j}\dots \gamma_{x+1}\in \mathcal S_T$ for all $i-1\geq j \geq x$.
    It follows that $\Mc(\mathcal S_T,\gamma_x)$ proceeds along direct
    arrows to $v$.  From that point on, by the defining condition for $v$, the
    string constructed so far will not lie in $\mathcal S_T$, and so it
    will continue by inverse arrows away from $v$.

    Now consider the extension of $\gamma_x$ in the opposite direction.  If
    the segment $\gamma_{x-1}\cdots \gamma_{y}$ were in $\mathcal S_T$, then
    the extension of it and $\Omega$ would be too, violating the minimality
    of $x$.  (If $\Omega$ is not defined, then $\gamma_{x-1}\cdots\gamma_y$
    itself violates the minimality of $x$.)  Thus, the extension from
    $s(\gamma_{x})$ proceeds entirely by direct arrows.  This completes the
    proof. 
  \end{proof}

  Again, as for the Ext-projectives of $T$, this shows that for each injective
  string $I_v$ corresponding to a vertex $v$ outside the support of $T$,
  there are
  two arrows $\alpha_1,\alpha_2$ of opposite orientations such that  $I_v=
  \Mc(\mathcal S_T,\alpha_1)=\Mc(\mathcal S_T,\alpha_2)$.

  Finally, we consider the injective strings at sink fringe vertices.

  \begin{proposition} Let $v$ be a sink fringe vertex.  Let
    $I_v=\gamma_r\dots\gamma_0$ be the injective at $v$, with $e(\gamma_r)=v$,
    which is composed
    entirely of direct arrows.  Let $x$ be minimal such that $\gamma_{r-1}\dots\gamma_{x+1}$ is contained in $\mathcal S_T$.  Then $I_v=\Mc(\mathcal S_T,\gamma_x)$.
  \end{proposition}
  \begin{proof} The proof is the same as for Proposition \ref{inj}, except
    that since $v$ is a fringe vertex, once the string reaches $v$, it ends.
    \end{proof}

\begin{proposition} No strings are constructed with greater multiplicity
  than given in Theorem \ref{cang}.  In particular, those strings not
  listed there, are not $\Mc(\mathcal S_T,\alpha)$ for any arrow
  $\alpha$ of $\hat Q$.  \end{proposition}

\begin{proof}
  The number of arrows in $\hat Q$ can be determined by counting the number
  of arrows going to each vertex.  This is the number of sink fringe vertices of
  $\hat Q$ plus twice the number of vertices of
  $Q$.  Since the Ext-projectives of $T$ are support $\tau$-tilting, the
  number of Ext-projectives of $T$ equals the number of vertices in the
  support of $T$; this plus the number of vertices not in the support of
  $T$ equals the number of vertices of $Q$.  Therefore, the number of
  strings listed in Theorem \ref{cang} equals the total number of arrows in
  $\hat Q$.  Therefore, each string is constructed with exactly the
  multiplicity given in Theorem \ref{cang}.  
   \end{proof}  

This completes the proof of Theorem \ref{cang}.

\begin{example}
Consider the gentle algebra $A=kQ/I$, where $Q=A_2$ and $I=0$, with the fringed bound quiver $(\hat{Q}, \hat{I})$ as below.
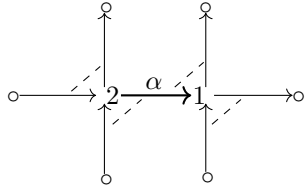
\begin{center}
\begin{tikzpicture}[scale=0.55]
---------
\draw [thick,->] (0.2,0) --(1.9,0);
\node at (0,0) {$2$};
\node at (2.1,0) {$1$};
\node at (1,0.3) {$\alpha$};

---------
\node at (-2.4,-0.05) {$\circ$};
\draw [->] (-2.25,0) -- (-0.4,0);
\draw [->] (-0.2,0.2) -- (-0.2,2);
\node at (-0.15,2.1) {$\circ$};
\draw [->] (-0.2,-1.9) -- (-0.2,-0.2);
\node at (-0.15,-2.05) {$\circ$};
\draw [dashed] (-1,0.1) -- (-0.3,0.7);
\draw [dashed] (0,-0.7) -- (0.7,-0.1);
---
\draw [->] (2.25,0.2) -- (2.25,2);
\node at (2.25,2.1) {$\circ$};
\draw [->] (2.45,0) --(4.35,0);
\node at (4.5,-0.05) {$\circ$};
\draw [->] (2.25,-1.85) --(2.25,-0.2) ;
\node at (2.3,-2) {$\circ$};
\draw [dashed] (1.5,0.2) -- (2.2,0.8);
\draw [dashed] (2.4,-0.7) -- (3.1,-0.1);
-------
\end{tikzpicture}
\end{center}

For the (functorially finite) torsion class $T=$ mod $A$, the set of strings $\mathcal{S}_T= \{e_1,e_2,\alpha \}$ is drawn in bold.

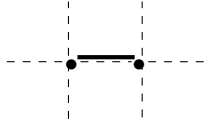
\begin{center}
\begin{tikzpicture}[scale=0.4]
---------
\draw [dashed] (0.2,0) --(1.9,0);
---------
\draw [-,dashed] (-2.25,0) -- (-0.4,0);
\draw [-,dashed] (-0.2,0.2) -- (-0.2,2);
\draw [-,dashed] (-0.2,-1.9) -- (-0.2,-0.2);
---
\draw [-,dashed] (2.25,0.2) -- (2.25,2);
\draw [-,dashed] (2.45,0) --(4.35,0);
\draw [-,dashed] (2.25,-1.85) --(2.25,-0.2) ;
-------
\draw [ultra thick] (0.1,0.15) --(2,0.15);
\node at (-0.1,-0.1) {$\bullet$};
\node at (2.15,-0.1) {$\bullet$};
\end{tikzpicture}
\end{center}
We now verify that as $\alpha$ runs through the arrows of $\hat Q$, we indeed obtain exactly the strings given in Theorem \ref{cang}.  The strings of the form $\cohook(M)$ for $M$ an Ext-projective indecomposable, are as follows.  Each is generated twice, by the arrows drawn as double lines.  

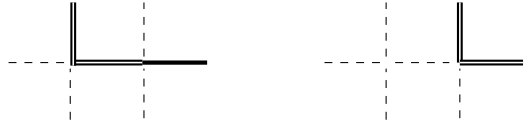
\begin{center}
\begin{tikzpicture}[scale=0.4]
\draw [-,dashed] (-2.25,0) -- (-0.4,0);
\draw [-,dashed] (-0.2,-1.9) -- (-0.2,-0.2);
---
\draw [-,dashed] (2.25,0.2) -- (2.25,2);
\draw [-,dashed] (2.25,-1.85) --(2.25,-0.2) ;
-------
\draw [-,thick, double] (-0.1,0) --(2.2,0);
\draw [-,thick, double] (-0.1,-0.1) -- (-0.1,2);
\draw [-,ultra thick] (2.2,0) --(4.35,0);
\end{tikzpicture}
\qquad
\qquad
\begin{tikzpicture}[scale=0.4]
\draw [-,dashed] (0.1,0) --(1.9,0);
\draw [-,dashed] (-2.25,0) -- (-0.4,0);
\draw [-,dashed] (-0.2,0.2) -- (-0.2,2);
\draw [-,dashed] (-0.2,-1.9) -- (-0.2,-0.2);
\draw [-,dashed] (2.25,-1.85) --(2.25,-0.2) ;
---
\draw [-,thick,double] (2.25,0) -- (2.25,2);
\draw [-,thick,double] (2.25,0) --(4.35,0);
-------
\end{tikzpicture}
\end{center}
Moreover, the injective strings corresponding to the fringe sink vertices are each generated once, as indicated by the double lines:

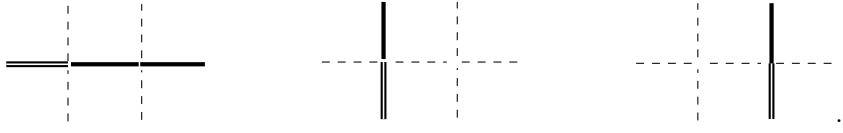
\begin{center}
\begin{tikzpicture}[scale=0.4]
\draw [-,dashed] (-0.2,0.1) -- (-0.2,2);
\draw [-,dashed] (-0.2,-1.9) -- (-0.2,-0.2);
\draw [-,dashed] (2.25,0.2) -- (2.25,2);
\draw [-,dashed] (2.25,-1.85) --(2.25,-0.2) ;
-------
\draw [-,ultra thick] (-0.1,0) --(2.15,0);
\draw [-,ultra thick] (2.2,0) --(4.35,0);
\draw [-,thick,double] (-2.25,0) -- (-0.2,0);
\end{tikzpicture}
\qquad
\qquad
\begin{tikzpicture}[scale=0.4]
\draw [-,dashed] (0.2,0) --(1.9,0);
\draw [-,dashed] (-2.25,0) -- (-0.4,0);
\draw [-,dashed] (2.25,-1.85) --(2.25,-0.2) ;
\draw [-,dashed] (2.25,0.2) -- (2.25,2);
\draw [-,dashed] (2.4,0) --(4.35,0);
------
\draw [-,thick,double] (-0.2,-1.9) -- (-0.2,0);
\draw [-, ultra thick] (-0.2,0.1) -- (-0.2,2);
\end{tikzpicture}
\qquad
\qquad
\begin{tikzpicture}[scale=0.4]
\draw [-,dashed] (0.2,0) --(1.9,0);
\draw [-,dashed] (-2.25,0) -- (-0.4,0);
\draw [-,dashed] (-0.2,0.2) -- (-0.2,2);
\draw [-,dashed] (-0.2,-1.9) -- (-0.2,-0.2);
\draw [-,dashed] (2.4,0) --(4.35,0);
---------
\draw [-,thick,double] (2.25,-1.85) --(2.25,0) ;
\draw [-, ultra thick] (2.25,0.0) -- (2.25,2);
\end{tikzpicture}.
\end{center}
\end{example}

\section{Extension groups of strings}\label{Ext}

The aim of this section is to study the $\Ext^1$-groups of string modules over gentle algebras. By the Auslander-Reiten formula (see Theorem \ref{ARformula}), this amounts to studying the quotient of a Hom-space by morphisms factoring through an injective.\\

Although $\Hom_{A}(X, \tau_{A}Y)$ for a gentle algebra $A$ admits a nice basis, given by graph maps as in Theorem \ref{Thm-graph-maps}, it appears difficult to decide which of these could form a basis for the quotient by $I_{A}(X, \tau_{A}Y)$. This question turns out to be easier to answer in the fringed algebra case, and since we have the following isomorphism 
\begin{equation}\Ext^1_{A}(Y, X) \simeq \Ext^1_{\hat{A}}(Y, X)
\end{equation}\label{Hom-equality}
from Proposition \ref{tau-cohook}, it is enough to study fringed algebras.

In the following, we first study graph maps starting in an injective string. We give an explicit combinatorial description of the graph maps in $\Hom_{\hat{A}}(X, \tau_{\hat A}Y)$ that factor through injectives in Subsection \ref{graph-maps-factoring}. Then we use that information to determine a basis for $\Ext^1_{A}(Y,X)$ in terms of short exact sequences in Subsection \ref{basis-ext}.

\subsection{Graph maps from an injective module}
We first study graph maps starting in an injective string $I_t$ over $A$. Such an injective string module can be written as $I_t= (\alpha^{-1}_{n+m}\cdots\alpha^{-1}_{n+1})\cdot_t(\beta_{n}\cdots\beta_{1})$, for a vertex $t\in Q_0$ and some arrows $\alpha_i , \beta_j \in Q_1$. Here we write $\cdot_t$ rather than $\cdot$ to indicate that the paths we compose meet at $t$.

The following notion turns out to be central:
A graph map $f_T$ given by an admissible pair $T=((F_1,E,D_1),(F_2,E,D_2))$ is called {\em two-sided} if at least one of $D_1$ and $D_2$ has positive length, and the same holds for $F_1$ and $F_2$.

\begin{lemma}\label{start-in-injective}
A graph map $f_T: I_t \to Y$ starting in an injective string $I_t$ is not two-sided.
\end{lemma}

\begin{proof}
There is no quotient factorization $ (F'_1\theta, E, \gamma^{-1} D'_1)$ of $I_t$ with arrows $\gamma,\theta$ in $Q$ because an inverse arrow cannot precede a direct arrow  in $I_t$.
Thus, in any quotient factorization $(F_1, E, D_1)$ of $I_t$, one of $D_1$ or $F_1$ has length zero. We consider, without loss of generality, the case where $D_1$ has length zero, thus $I_t=F_1E$.
Since $I_t$ is injective, there is no arrow $\beta$ in $Q$ such that $E\beta$ is a string. Hence $Y$ cannot admit a submodule factorization of the form $(F_2, E, \beta D'_2)$. We conclude that $D_2$ has length zero in any submodule factorization of $Y$ of the form $(F_2, E, D_2)$. So, $f_T$ is not two-sided.
\end{proof}

\subsection{Graph maps factoring through an injective string module}\label{graph-maps-factoring}

For a pair of strings $X$ and $Y$ in $A$, we now investigate which graph maps in $\Hom_{\hat{A}}(X,\tau_{\hat{A}} Y)$ factor through an injective string $I_t$ in the fringed algebra $\hat{A}$. From Proposition~\ref{tau-cohook}, we know that the Auslander-Reiten translation of a string $Y$ in $\hat{A}$ is obtained by  cohook completion: $$\cohook(Y)= \tau_{\hat{A}} Y = \; _c\!Y_c = \mathcal{I}\cdot\beta\cdot Y\cdot\alpha^{-1}\cdot\mathcal{D}$$
where $\mathcal I$ and $\mathcal D$ respectively denote inverse and direct paths in $\hat{Q}$:\\

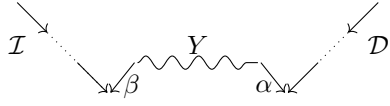
\begin{center}
\begin{tikzpicture}[scale=0.8]
\draw  [-,decorate,decoration=snake] (0,0) --(2,0);
\node at (1,0.3) {$Y$};

\draw [->] (-0.05,0) -- (-0.45,-0.5);
\node at (-0.1,-0.4) {$\beta$};
\draw [->] (-1,0) -- (-0.5,-0.5);
\draw [dotted] (-1.5, 0.5) -- (-0.5,-0.5);
\draw [->] (-2,1) -- (-1.5, 0.5);
\node at (-2,0.3) {$\mathcal{I}$};
----
\draw [->] (2.05,0) -- (2.45,-0.5);
\node at (2.1,-0.4) {$\alpha$};
\draw [->] (3,0) -- (2.5,-0.5);
\draw  [dotted] (3.5,0.5) -- (3,0);
\draw [->] (4,1) -- (3.5,0.5);
\node at (4,0.3) {$\mathcal{D}$};

\end{tikzpicture}
\end{center}

We say a string $E$ {\em lies on one of the arms of $\cohook(Y)$} if $E$ is a substring (possibly of length zero) of $\mathcal I$ or $\mathcal D$.
Moreover, we call $Y$  \emph{connectable} to  $X$ if there exists an arrow $\alpha$ such that $Y \alpha^{-1} X$ is a string in $Q$, or dually, $X \alpha Y$ is a string in $Q$:
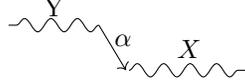
\begin{center}
\begin{tikzpicture}[scale=0.8]
\draw  [-,decorate,decoration=snake] (0,0.25) --(2,0.25);
\node at (1,0.6) {$X$};

\draw [->] (-0.5,1) --(-0.05,0.25);
\node at (-0.1,0.75) {$\alpha$};
----

\draw  [-,decorate,decoration=snake] (-0.5,1) --(-2,1);
\node at (-1.25,1.3) {Y};
\end{tikzpicture}
\end{center}
\bigskip
Note that if   $Y$ is connectable to  $X$, there exists a graph map $f_T \in \Hom_{\hat{A}}(X,\tau_{\hat{A}} Y)$ given by an admissible pair $T=((F_1,E,D_1),(F_2,E,D_2))$ where   $E$ lies on one of the arms of $\cohook(Y)$:
Denote by $E$ the longest direct string at the end of $X$ and write $X=ED_1$.
This gives, with $F_1$ of length zero, the quotient factorization $(F_1,E,D_1)$ of $X$.
By the cohook construction, $\mathcal{D}$ is the longest direct string in $\hat{A}$ ending in $e(\alpha).$ But $E$ is a direct string ending in $e(\alpha)$. Thus $E$ is a submodule of $\mathcal{D}$, and we can write $\mathcal D = ED_2$. Setting 
$F_2 =  \mathcal{I}\cdot\beta\cdot Y\cdot\alpha^{-1}$
gives the submodule factorization
$(F_2,E,D_2)$ of $\tau_{\hat{A}} Y.$

We refer to this graph map $f_T$ as a \emph{connecting map}. The following is an illustration of the factorization $T$ where the cohooks of are drawn dashed. 
\bigskip
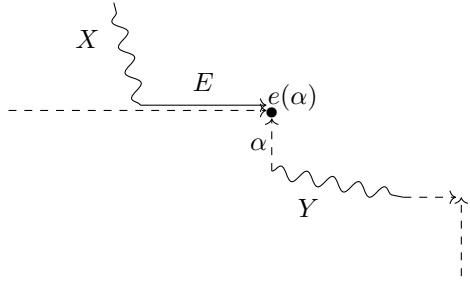
\begin{center}
\begin{tikzpicture}[scale=0.7]


\draw [->] (-2.5,1.25) -- (-0.1,1.25);
\draw [-,decorate,decoration=snake] (-2.5,1.25) -- (-3,3.2);

\node at (-3.5,2.5) {$X$};
\node at (-1.3,1.7) {$E$};
-----


\draw [->,dashed] (-5,1.15) --(-0.1,1.15);
\draw [->,dashed] (0,0) -- (0,1);
\draw [-,decorate,decoration=snake] (0,0) --(2.5,-0.5);
\draw [->,dashed] (2.5,-0.5) -- (3.5,-0.5);
\draw [->,dashed] (3.6,-2) -- (3.6,-0.5);

\node at (-0.25,0.5) {$\alpha$};
\node at (0.7,-0.7) {$Y$}; 
\node at (0,1.1) {$\bullet$};
\node at (0.4,1.4) {$e(\alpha)$};

\end{tikzpicture}
\end{center}
\bigskip

\begin{theorem}\label{factor-through-injective}
For strings $X$ and $Y$ in $Q$, let $f_T \in \Hom_{\hat{A}}(X,\tau_{\hat{A}} Y)$ be a graph map given by an admissible pair $T=((F_1,E,D_1),(F_2,E,D_2))$. 
If $E$ lies on one of the arms of $\cohook(Y)$ and  $f_T$ is not a connecting map, then $f_T$ factors through an injective string of mod-$\hat{A}$.

\end{theorem}

\begin{proof}

Consider a graph map  $f_T \in \Hom_{\hat{A}}(X,\tau_{\hat{A}} Y)$ where  $E$ lies on one of the arms of $\cohook(Y)$. Without loss of generality, assume $E$ is a substring of $\mathcal{D}$.  Since $(F_2,E,D_2)$ is a submodule factorization of $\cohook(Y)$, the string $E$ lies on the end of the direct string $\mathcal{D}$, so we have   $\mathcal{D}=E D_2$.  We want to find an injective string $I_t$ and graph maps $f_{T'}\in \Hom_{\hat{A}}(X,I_t)$ and $f_{T''}\in \Hom_{\hat{A}}(I_t,\tau_{\hat{A}} Y)$ such that $f_T=f_{T''} \circ f_{T'}$.

In the factorization $(F_1,E,D_1)$ of the string $X$, write $F_1= F_1'\gamma_c \cdots \gamma_1$ where $F_1'$ does not start with a direct arrow. The assumptions of the theorem imply that such a direct string $\gamma_c \cdots \gamma_1$ always exists: If $F_1$ has length zero, then $Y$ is  connectable to  $X$ via the arrow $\alpha$, and $f_T$ is the corresponding connecting map. Otherwise, if $F_1$ starts with an inverse arrow, then $E$ is not a factor module of $X$ and  $(F_1,E,D_1)$ not a quotient factorization. Therefore, the direct arrow $\gamma_c$ exists in $X$.
Denote by $t$ its end point, see the following figure:

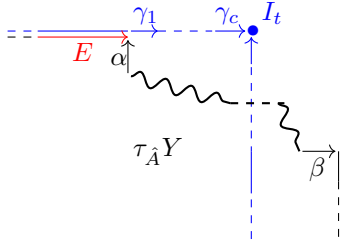
\begin{center}
\begin{tikzpicture}[scale=0.8]
\draw [blue] [dashed] (-2.5,0)--(-2,0);
\draw [blue] (-2,0) -- (-0.5,0);
\draw [blue] [->] (-0.45,0) -- (0,0);
\node [blue] at (-0.2,0.25){$\gamma_1$};
\draw [blue] [dashed] (0,0) -- (0.9,0);
\draw [blue] [->] (0.95,0) -- (1.45,0);
\draw [blue][->] (1.55,-0.55) -- (1.55,-0.1);
\draw [blue] [dashed] (1.55,-0.6) -- (1.55,-2);
\draw [blue] (1.55,-2) -- (1.55,-2.7);
\draw [blue] [dashed] (1.55,-2.8) -- (1.55,-3.5);
\node [blue] at (1.15,0.25){$\gamma_c$} ;
\node [blue] at (1.58,0) {$\bullet$} ;
\node [blue] at (1.9,0.3) {$I_t$} ;
-------
\draw [red][->] (-2,-0.1) -- (-0.5,-0.1);
\node [red] at (-1.25,-0.35){$E$} ;
-------
\draw [->] (-0.5,-0.7)-- (-0.5,-0.15);
\node at (-0.65,-0.47){$\alpha$} ;
\draw [dashed] (-2.5,-0.1)-- (-2,-0.1);
\draw  [thick,decorate,decoration=snake](-0.45,-0.75) -- (1.2,-1.2);
\draw  [thick][dashed] (1.2,-1.2) -- (2.1,-1.2);

\draw  [thick,decorate,decoration=snake] (2,-1.2)-- (2.35,-2);
\draw [->] (2.4,-2) -- (2.95,-2);
\draw (3,-2) -- (3,-2.5);
\draw [dashed] (3,-2.5) -- (3,-3.5);
\node at (2.65,-2.3) {$\beta$} ; 
-------
\node at (0,-2) {$\tau_{\hat{A}}Y$} ;

\end{tikzpicture}
\end{center}

Set $E'=\mathcal{I'}\gamma_c \cdots \gamma_1 E$, where $\mathcal{I'}$ is the longest inverse substring in $X$ immediately following $t$.
Then $E'$ is a factor module of $X$ and a submodule of $I_t$, therefore there exists a graph map $f_{T'}\in \Hom_{\hat{A}}(X,I_t)$ with $E'$ as middle term of the factorizations.
Similarly, $E''= \mathcal{D}$ is a factor module of $I_t$ and a submodule of $\tau_{\hat{A}} Y$, therefore there exists a graph map $f_{T''}\in \Hom_{\hat{A}}(I_t,\tau_{\hat{A}} Y)$ with $E''$ as middle term.
By construction, we get $f_T=f_{T''} \circ f_{T'}$.
\end{proof}

In the following corollary we describe a basis of $I_{\hat{A}}(X,\tau_{\hat{A}} Y)$ in terms of graph maps:

\begin{corollary}\label{basis-for-I}
Let $X$ and $Y$ be two strings in $Q$.
Then a basis for $I_{\hat{A}}(X,\tau_{\hat{A}} Y)$ is formed by the graph maps $ f_T \in \Hom_{\hat{A}}(X,\tau_{\hat{A}} Y)$, with
 $T=((F_1,E,D_1),(F_2,E,D_2)) $, such that $E$ lies on one of the arms of $\cohook(Y)$ and $f_T$ is not a connecting map.
\end{corollary}

\begin{proof}
Theorem \ref{factor-through-injective} assures that each graph map $f_T$ belongs to $I_{\hat{A}}(X,\tau_{\hat{A}} Y)$.
We show that every morphism $f \in \Hom_{\hat{A}}(X,\tau_{\hat{A}} Y)$ that factors through an injective $I$ is a linear combination of the graph maps as described in the hypothesis. Since graph maps are linearly independent, we therefore get a basis for $I_{\hat{A}}(X,\tau_{\hat{A}} Y)$.

Thus, assume  $f=g \circ h$ where $h \in \Hom_{\hat{A}}(X,I)$ and $g \in \Hom_{\hat{A}}(I,\tau_{\hat{A}} Y)$.
Write the maps $g$ and $h$ in the basis of graph maps, thus  $g=\sum a_i g_i$ and  $h=\sum b_j h_j$ with scalars $a_i, b_j \in k$ and graph maps $g_i,h_j$.  
Each $f_{ij}=g_i\circ h_j$ is either zero or a composition of graph maps (which is again a graph map).  It suffices to show that the $f_{ij}$ satisfy the conditions stated above.

The graph map $f_{ij}$  factors  via $h_j$ and $g_i$ through some injective string $I_t$ (which is a direct summand of $I$).
In general, if a composition of graph maps, given by  factorizations $T'=((F'_1,E',D'_1),(F'_2,E',D'_2)) $ and  $T''=((F''_1,E'',D''_1),(F''_2,E'',D''_2)) $, yields a graph map given by  the factorization $T=((F_1,E,D_1),(F_2,E,D_2)) $, one necessarily has $E \subset E'$ and $E \subset E''$.
Since $g_i$ starts in $I_t$, lemma \ref{start-in-injective} implies that the factorization $T''$ defining the graph map  $g_i$ cannot be two-sided.
Therefore, without loss of generality, we can assume that both $D''_1$ and $D''_2$ have length zero, that is, we have $I_t =F''_1E''$ and $\tau_{\hat{A}} Y= F''_2E''$.
If $E''$ is a proper quotient of  $I_t$, we conclude that $E''$ is a direct string, properly contained in the arm $\mathcal{D}$ of $\tau_{\hat{A}} Y=\cohook(Y)$. This implies that  $E$ lies on one of the arms of $\cohook(Y)$ and $f_{ij}$ is not a connecting map.
Otherwise,  if $E''=I_t$, we get that $\tau_{\hat{A}}Y= F''_2E''$ contains $e(E'')$ as internal vertex, which is impossible since this does not belong to $Q$, or else $\tau_{\hat{A}}Y= E''=I_t$ which is also impossible since an injective is not an image under the Auslander-Reiten translate.

\end{proof}

\subsection{The vector space $\Ext^1_{A}(Y,X)$}\label{basis-ext}

Recall from Theorem \ref{ARformula} and  Proposition \ref{tau-cohook}  that we have the following isomorphisms:

\begin{equation}\label{duality}
\Ext^1_{A}(Y,X) \simeq \Ext^1_{\hat{A}}(Y, X) \simeq D(\Hom_{\hat{A}}(X,\tau_{\hat{A}} Y)/I_{\hat{A}}(X,\tau_{\hat{A}} Y))
\end{equation}
In Subsection \ref{graph-maps-factoring} we determined a basis for this space given by the graph maps in $\Hom_{\hat{A}}(X,\tau_{\hat{A}} Y)$ that do not belong to  $I_{\hat{A}}(X,\tau_{\hat{A}} Y)$. In Corollary \ref{basis-for-I} we gave a precise description of which graph maps have to be excluded. We aim in this subsection for a description of a basis for $\Ext^1_{A}(Y,X)$ in terms of short exact sequences.

Clearly, every connecting graph map $f_T$ gives rise to a non-zero extension 
\begin{equation}\label{epsilonT}
\epsilon_T: \quad 0 \rightarrow X \rightarrow Y \alpha^{-1} X \rightarrow Y \rightarrow 0
\end{equation}
Moreover, a two-sided graph map 
$$f_T \in \Hom_{A}(X, Y) \mbox{ with } T=((F_1,E,D_1), (F_2,E,D_2)) $$
yields a non-zero extension $\epsilon_T$ as follows, see \cite{S}:

\begin{equation}\label{extension-epsilon}
\epsilon_T : \quad 0 \rightarrow X \rightarrow F_1ED_2 \oplus F_2ED_1 \rightarrow Y \rightarrow 0
\end{equation}
\bigskip

\begin{lemma}\label{ext-lemma}
For strings $X$ and $Y$ in $Q$, there is a bijection between the 
two-sided graph maps 
$f_T \in \Hom_{A}(X, Y)$ and the non-connecting  graph maps in 
$\Hom_{\hat{A}}(X,\tau_{\hat{A}} Y)$ that do not belong to $I_{\hat{A}}(X,\tau_{\hat{A}} Y)$
\end{lemma}

\begin{proof}
Assume $f_T \in \Hom_{A}(X, Y)$ is a two-sided graph map given by the factorization  $T=((F_1,E,D_1),(F_2,E,D_2)) $. The following construction associates to $f_T$ a graph map $f_{\hat{T}} \in \Hom_{\hat{A}}(X,\tau_{\hat{A}} Y)$ given by $\hat{T}=((\hat{F}_1,\hat{E},\hat{D}_1),(\hat{F}_2,\hat{E},\hat{D}_2)) $: As before, we denote $\tau_{\hat{A}} Y =  \mathcal{I}\beta Y\alpha^{-1} \mathcal{D}$. 
If $F_2$ has positive length, we define $\hat{F}_2 = \mathcal{I}\beta F_2$ and leave $E$ unchanged.
Otherwise, since $T$ is two-sided, we know that $F_1$ has positive length. As $E$ is a submodule, the string $F_1$ starts with a direct arrow, and this must be the arrow $\beta$ since in the gentle algebra $A$ there is only one way to extend the string $E$. 
We subdivide $F_1$ as  $F_1=\hat{F}_1\mathcal{I'}\beta $, where $\mathcal{I'}$ is the longest inverse substring in $X$ immediately following $\beta$, and extend $E$ at the end by $\mathcal{I'}\beta$, that is, we put $\hat E= \mathcal{I'}\beta E$.
We proceed in the same way for the $D$ side. 
These basis elements $f_{\hat{T}} \in \Hom_{\hat{A}}(X,\tau_{\hat{A}} Y)$ are clearly distinct, non-connecting and do not lie in $I_{\hat{A}}(X,\tau_{\hat{A}} Y)$, by the description given in corollary \ref{basis-for-I}.
It is moreover easy to see that every such basis element $f_{\hat{T}} \in \Hom_{\hat{A}}(X,\tau_{\hat{A}} Y)$ is obtained by this construction. 
\end{proof}

\begin{theorem}\label{basis-for-ext}
For strings $X$ and $Y$ in $Q$, the extensions $\epsilon_T$ given in Equation \eqref{epsilonT} and \eqref{extension-epsilon}, with $T$ connecting or two-sided, form a basis for $\Ext^1_{A}(Y,X)$.
\end{theorem}
\begin{proof}
We know from the isomorphism in \eqref{duality} and the bijection in Lemma \ref{ext-lemma} that we listed the correct number of elements $\epsilon_T$ in $\Ext^1_{A}(Y,X)$.
It is therefore sufficient to show that they are linearly independent. 
We do so by showing that the $\epsilon_T$ correspond to linearly independent elements in an isomorphic space.

Consider an injective envelope $I$ of $X$ and extend to a short exact sequence:

\[ 
\xymatrix{
0  \ar[r] &  X \ar[r]^{\iota} 
& I \ar[r]^{\pi}  & Z \ar[r] & 0}
\]

Applying the functor $\Hom(Y, -)$ yields the exact sequence

\[ 
\xymatrix{
0  \ar[r] &  \Hom(Y,X) \ar[r]^{\iota^*} 
& \Hom(Y,I) \ar[r]^{\pi^*}  & \Hom(Y,Z) \ar[r]^\delta & 
\Ext(Y,X) \ar[r] & 0}
\]
where we use $\Ext(Y,I)=0 $ since $I$ is injective.
Thus, $\Ext(Y,X)$ is isomorphic to $\Coker \pi^*$, with isomorphism induced by the the connecting homomorphism $\delta$.
To study the map $\delta$, denote by $s_1, \ldots , s_n$ the sinks of the string $X$. Then $I = I_{s_1} \oplus \cdots \oplus I_{s_n}$.

Assume the extension $\epsilon_T$ is given by a connecting graph map as in Equation \eqref{epsilonT}. 
Denote by $E$ the longest inverse substring of $Y$ such that $E\alpha^{-1}$ is a string. Then $E$ is a factor module of $Y$, and it also a factor module of $I_{s_n}$ and a submodule of the summand of $Z$ induced by $I_{s_n}$. The corresponding graph map $f_E \in \Hom(Y,Z)$ clearly satisfies $\delta(f_E) = \epsilon_T$.

Moreover, denote by $t_1, \ldots , t_m$ the sources of the string $X$ that have two arrows in $X$ attached to it. Then the socle of the module $Z$ contains the same vertices $t_1, \ldots , t_m$.
Let now $f_T \in \Hom_{A}(X, Y)$ be a two-sided graph map given by the factorization  $T=((F_1,E,D_1),(F_2,E,D_2)) $, and let $t_i, \ldots , t_j$ be the sources of $X$ that are contained in the string $E$. Note that different factorizations $T,T'$ give rise to different intervals $[i,j],[i',j'] $. 
We let $f_E \in \Hom_{A}(Y, Z)$ be the map identifying the elements $t_k$ of $Y$ with the corresponding elements $t_k$ in $Z$, for  $k \in [i,j]$. 
It is an exercise in linear algebra to see  that $\delta(f_E) = \epsilon_T$. 

The maps $f_E$ stemming from a two-sided graph map are linearly independent  from the maps induced by connecting graph maps, and they are linearly independent amongst each other: this comes from the fact that the middle factors $E,E'$ of two-sided graph maps $f_T,f_{T'}$ cannot properly overlap. In fact, the only way that $E$ and $E'$ intersect non-trivially is that one is contained in the other. 

Moreover, the classes of $f_E$ are  linearly independent in $\Coker \pi^*$, since graph maps factoring through an injective $I$ are described in
Corollary  \ref{basis-for-I}, and these cannot alter the linear independence for the graph maps $f_E$.
Therefore the same holds for the short exact sequences $\epsilon_T$.
\end{proof}

\section{Uniqueness of kisses between exchangeable modules}\label{uniqueness}

Let $A$ be a gentle algebra. Let $Z$ be an almost complete support $\tau$-tilting $A$-module. Then by~\cite[Theorem 2.18]{AIR} we know that there exist exactly two support $\tau$-tilting $A$-modules, say $T$ and $T'$, having $Z$ as a direct summand. 
We recall from Section~\ref{torposet} that in this case, one of $\Fac T$ and $\Fac T'$ covers the other in the poset of functorially finite torsion classes.  Supposing $\Fac T$ covers $\Fac T'$, we say that $T'$ is the left mutation of $T$, and we write $T'=\mu^{-}_X(T)$. 
Let us write $T=X\oplus Z$ and $T'=Y\oplus Z$ where $X$ is a $\tau$-rigid indecomposable module, and $Y$ is a $\tau$-rigid indecomposable module which is not isomorphic to $X$ or else zero. 


Consider the maximal non-kissing collections corresponding to $T$ and $T'$.  There is
a string $C=\cohook(X)$ in the maximal non-kissing collection corresponding to $T$.  In
the maximal non-kissing collection corresponding to $T'$, this string is replaced by another
string, $D$.  We call such a pair of strings \emph{exchangeable}.  As shown in Section \ref{torposet}, there will be one or more kisses from 
$C$ to $D$.  In his more restricted setting, McConville showed that exchangeable strings kiss exactly once \cite[Theorem 3.2(3)]{McC}.  It is natural to ask whether this, like so many of the other results from 
\cite{McC}, extends verbatim to the general gentle setting.  

It turns out that 
this is not the case.   
For example, consider the gentle algebra given by the quiver 

\[ \xymatrix{
 1\ar@(ur,ul)_{\alpha}[] &2\ar[l]
}\]
with the relation $\alpha^2=0$. Then, the direct sum of the projectives, $P_1 \oplus P_2$, is a $\tau$-rigid module, its mutation at $P_1$ is the module $\tau^-(P_1) \oplus P_2$ and one has that $\dim_k \Hom(P_1,\tau (\tau^- P_1))= \dim_k \Hom(P_1, P_1)=2$. This example comes from \cite[Section 13.6]{GLS}. We thank Gustavo Jasso for pointing out its relevance to our situation.

However, when the $\tau$-rigid indecomposable $X$ is replaced by a module $Y$ which is a brick, the next theorem will show that the corresponding strings kiss once, i.e., that the dimension of $\Hom(X,\tau Y)$ is indeed one. 

\begin{theorem}\label{exchange-modules}
If $Y\oplus Z$ is obtained by a left mutation from $T=X\oplus Z$, that is, $Y\oplus Z = \mu_X^- (T)$, and $Y$ is a brick, then $\dim_k \Hom_A(X,\tau Y)=1.$ 
\end{theorem}

Before proving the theorem we will explain the necessary background and a lemma which will be used in the proof of the theorem. Let \begin{displaymath}\xymatrix{
P_1\ar[r]^{p_1}& P_0 \ar[r]^{p_0}&X\ar[r]& 0 \\
Q_1\ar[r]^{q_1}& Q_0 \ar[r]^{q_0}&Y\ar[r]& 0 \\
R_1\ar[r]^{r_1}& R_0 \ar[r]^{r_0}&Z\ar[r]& 0 \\}
\end{displaymath} be minimal projective presentations, and we denote by $P, Q,$ and $R$ the corresponding two-term complexes of projective modules in $\kbp$.

We are going to prove some results in this setting before proving the main theorem of this section.
\begin{lemma}\label{lmm} We have the following properties; \begin{enumerate}
\item $\dim_k \Hom_{\kbp} (Q,Q) = 1$\label{lmm1},
\item $\Hom_{\kbp} (Q,R[1]) = 0$\label{lmm2},
\item $\dim_k \Hom_{\kbp}(Q,P[1])=1$\label{lmm3}.
\end{enumerate}
\end{lemma}

\begin{proof}
(1) Let $f_{\bullet}, g_\bullet \in \Hom_{\kbp}(Q,Q)$. By the universal property of cokernels, they induce the morphisms $f,g\in \Hom_A(Y,Y)$ such that $fq_0=q_0f_0$ and $gq_0=q_0g_0$. Since $Y$ is a brick module, we know that $\dim_k \Hom_{A}(Y,Y)=1$, so there exists a $\lambda\in k$ such that $f-\lambda g=0$.

Denote by $Q_\bullet$ the  projective resolution of $Y$ obtained by completing $Q$. We get the following commutative diagram
\[\xymatrix@=30pt{
Q \ar[d]^{q_0}\ar[r]^{f_\bullet-\lambda g_\bullet} & Q_\bullet \ar[d]^{q_0}\\
Y\ar[r]^{f-\lambda g \:=\: 0}& Y&. 
}
\]

Finally, by homotopy uniqueness of projective resolutions \cite[Theorem 2.2.6]{W} one can easily prove that $f_\bullet-\lambda g_\bullet$ is null-homotopic.  

(2) Following \cite[Lemma 3.4]{AIR}, this is true if and only if $\Hom_A(Z,\tau Y)=0$. Now, since $Y\oplus Z$ is $\tau$-tilting, $\Hom_A (Y\oplus Z,\tau (Y\oplus Z))=0$, which yields the result.

(3) By~\cite[Theorem 3.2 and Corollary 3.9]{AIR} we can write the exchange of $X$ and $Y$ on the level of two-term complexes.  By~\cite[Definition-Proposition 1.7]{AIR}, this guarantees that there exists a triangle in $\kbp$ : 
\[\xymatrix{P\ar[r]&R'\ar[r]&Q\ar[r]&P[1]}\]
with $R'\in \add R.$ 

Applying the functor $\Hom_{\kbp}(Q,-)$, we get the long exact sequence
\[\xymatrix{...\ar[r]&\Hom(Q,Q)\ar[r]&\Hom(Q,P[1])\ar[r]&\Hom(Q,R'[1])\ar[r]&...}.\]

Using parts (\ref{lmm1}) and (\ref{lmm2}) and the fact that $R'[1]\in \add R[1]$, we now get that $\dim_k \Hom_{\kbp}(Q,P[1])$ is either $0$ or $1$. Assume it is $0$, then by \cite[Lemma 3.4]{AIR}  we have that $\Hom(X,\tau Y)=0$. This leads a contradiction: since $X$ and $Y$ are not compatible, $ \Hom_A(X,\tau Y)$ or $ \Hom_A(Y,\tau X)$ is nonzero. However, $ \Hom_A(Y,\tau X)=0$. This follows from the definition of a left mutation \cite[Definition-Proposition 2.28]{AIR} which guarantees that $\Fac Y \subseteq \Fac Z \subseteq \Fac T$, the fact that $X$ is Ext-projective in $\Fac T$ and \cite[Proposition 1.2(a)]{AIR} which states that $\Hom_A(Y,\tau X) =0$ if and only if $\Ext^1_A(X,\Fac Y) = 0$. Thus, $\dim_k \Hom_{\kbp}(Q,P[1])$ has to be $1.$
\end{proof}

\begin{proof}[Proof of Theorem~\ref{exchange-modules}]
Rewriting slightly \cite[Proposition 2.4]{AIR}, we get an exact sequence
\[
\xymatrix{\Hom_A(Q_0,X)\ar[r]^{(q_1,X)}&\Hom_A(Q_1,X)\ar[r]& \Du \Hom_A(X,\tau Y)\ar[r]& 0.
}\]

This gives that $\dim_k \Hom_A(X,\tau Y)$ is equal to the dimension of the quotient of $\Hom_A(Q_1,X)$ by $\Image (q_1,X)$.

We now show that the dimension of this quotient is one. 

Let $f,g\in \Hom_A(Q_1,X)$. Since $Q_1$ is projective, there exists $\overline f, \overline g \in \Hom_A(Q_1,P_0)$ such that $p_0\overline f=f$ and $p_0\overline g=g$.

These functions give chain morphism in $\Hom_{\kbp}(Q,P[1])$, which is of dimension one by Lemma~\ref{lmm}$(3)$. It follows that there exists $h_0 \in \Hom_A(Q_0,P_0)$, $h_1\in \Hom_A (Q_1,P_1)$ and $\lambda\in k$ such that $\overline f - \lambda \overline g = p_1 h_1 + h_0 q_1$.

Composing the last equality with $p_0$, we get that \[f-\lambda g = p_0h_0q_1 = (q_1,X)(p_0h_0),\]
which is to say that $f-\lambda g = 0$ in the quotient space.

Since $ \Hom_A(X,\tau Y)$ is nonzero we conclude the desired result.
\end{proof}

This result supposes that the mutation produces a nonzero module $Y$. In general, the support of $Z$ can be smaller than the support of $X\oplus Z$, in which case $Z$ is itself a support $\tau$-tilting module, and left mutation of $X\oplus Z$ at $X$ yields $Z$.  In this case, let $v$ be the vertex in the support of $X$ which is not in the support of $Z$.  Denote by $P_v$ and $I_v$ the projective and injective modules at vertex $v$, respectively. 

\begin{proposition}\label{inj-kiss} Let $Z$ be a support $\tau$-tilting module obtained by a left mutation from $T=X\oplus Z$, and let $v$ be the vertex in the support of $X$ which is not in the support of $Z$.  Suppose that $P_v$ is a brick. Then $\dim_k \Hom_A(X,I_v)=1$.
\end{proposition}

\begin{proof} As before, define $P$ and $R$ to be the two-term complexes corresponding to minimal projective presentations of $X$ and $Z$.  
Let $Q$ be the two-term complex  $P_v \rightarrow 0$.  We now check that the statements of Lemma \ref{lmm} still hold.  (1) follows from the fact that $P_v$ is a brick.  (2) follows from $Q\oplus R$ being a silting complex~\cite[Definition 1.5]{AIR}.  We conclude as in the proof of Lemma \ref{lmm}(3) that $\dim_k\Hom_{\kbp}(Q,P[1])\leq 1$.   This means that 
$$\dim_k \Hom(P_v,P_0)/p_1(\Hom(P_v,P_1)) \leq 1.$$  Since $\Hom(P_v,P_0)/p_1 \Hom(P_v,P_1)=\Hom(P_v,X),$ we conclude that $$\dim_k \Hom(P_v,X)\leq 1.$$  Since 
$X$ is supported over the vertex $v$, the dimension is exactly one. Finally, since we have $\dim_k\Hom_A(P_v,X)=\dim_k \Hom_A (X, I_v)$, this yields the desired result.
\end{proof}

From the previous results in this section, we deduce a combinatorial corollary.

\begin{theorem}
Suppose that $A$ is a gentle algebra such that every $\tau$-rigid indecomposable $A$-module is a brick.  If two maximal non-kissing collections of long strings in $\hat Q$ differ by replacing one string by another, then these two strings kiss exactly once.\end{theorem}  

\begin{proof} 
By Theorem \ref{thm-maximal-non-kissing}, we know that there is a bijective correspondence between maximal non-kissing collections of long strings in $\hat Q$ and basic support $\tau$-tilting modules for $A$.  Let the basic support $\tau$-tilting modules corresponding to the given maximal non-kissing collections be $T$ and $T'$.  
Since the two maximal non-kissing collections differ by replacing a single string by another, the corresponding support $\tau$-tilting modules differ by a single mutation.  Without loss of generality, let $T'$ be the left mutation of $T$.  Let $C$ and $D$ be
the corresponding strings.  Theorem \ref{covers} shows that there are no kisses from 
$D$ to $C$, but there is at least one kiss from $C$ to $D$.  We wish to show that there is in fact exactly one kiss.  

Let $X$ be the $A$-module associated to the long string $C$ (i.e., $C=\cohook(X)$).  There are two possibilities regarding
$D$: either it also corresponds to an $A$-module, or else it is an injective string in $\hat Q$.  Suppose first that it corresponds to an $A$-module, say $Y$.  Since $Y$ is $\tau$-rigid, it is a brick by assumption, and we can apply Theorem \ref{exchange-modules} to conclude
that $\dim\Hom(X,\tau Y)=1$.  By Theorem \ref{dims}, it follows that $C$ and $D$ kiss exactly once.  

Now, we consider the possibility that $D$ is an injective string in $\hat Q$, corresponding to the vertex $v\in Q_0$.  In this case, Proposition \ref{inj-kiss} tells us 
that $\dim\Hom_A(X,I_v)=1$. Each kiss from $C$ to $D$ gives rise to such a
morphism, so there is only one kiss between $C$ and $D$ in this case as well.  
\end{proof}

This recovers the uniqueness of kisses shown in \cite{McC}, since in that setting, bricks, strings, and $\tau$-rigid indecomposable modules all coincide.  





\end{document}